\documentclass{article}

\usepackage{rotating}
\usepackage{tikz-cd}
\usepackage[all]{xy}
\usepackage{amssymb}
\usepackage{amsmath}
\usepackage{amsthm}
\newtheorem{thm}{Theorem}
\newtheorem{prop}{Proposition}
\newtheorem{lem}{Lemma}
\newtheorem{rem}{Remark}

\newtheorem{defi}{Definition}

\def\Top{\mathop{\rm Top}\nolimits}
\def\RTop{\mathop{\rm RTop}\nolimits}
\def\Sch{\mathop{\rm Sch}\nolimits}

\def\Pic{\mathop{\rm Pic}\nolimits}

\def\dim{\mathop{\rm dim}\nolimits}

\def\Tr{\mathop{\rm Tr}\nolimits}
\def\Cor{\mathop{\rm Cor}\nolimits}

\def\SmVar{\mathop{\rm SmVar}\nolimits}

\def\PSmVar{\mathop{\rm PSmVar}\nolimits}
\def\Var{\mathop{\rm Var}\nolimits}
\def\Hom{\mathop{\rm Hom}\nolimits}
\def\Spec{\mathop{\rm Spec}\nolimits}

\def\supp{\mathop{\rm supp}\nolimits}
\def\QPVar{\mathop{\rm QPVar}\nolimits}
\def\AnSp{\mathop{\rm AnSp}\nolimits}
\def\CW{\mathop{\rm CW}\nolimits}

\def\PVar{\mathop{\rm PVar}\nolimits}

\def\sing{\mathop{\rm sing}\nolimits}

\def\diff{\mathop{\rm diff}\nolimits}
\def\Cone{\mathop{\rm Cone}\nolimits}
\def\ad{\mathop{\rm ad}\nolimits}

\def\log{\mathop{\rm log}\nolimits}
\def\Diff{\mathop{\rm Diff}\nolimits}
\def\An{\mathop{\rm An}\nolimits}
\def\Ho{\mathop{\rm Ho}\nolimits}
\def\PSh{\mathop{\rm PSh}\nolimits}
\def\AnSm{\mathop{\rm AnSm}\nolimits}
\def\Tr{\mathop{\rm Tr}\nolimits}

\def\DA{\mathop{\rm DA}\nolimits}

\def\Fun{\mathop{\rm Fun}\nolimits}

\def\Bti{\mathop{\rm Bti}\nolimits}

\def\coker{\mathop{\rm coker}\nolimits}

\def\Cat{\mathop{\rm Cat}\nolimits}
\def\RCat{\mathop{\rm RCat}\nolimits}

\def\an{\mathop{\rm an}\nolimits}

\def\Shv{\mathop{\rm Shv}\nolimits}
\def\Proj{\mathop{\rm Proj}\nolimits}

\def\PSch{\mathop{\rm PSch}\nolimits}

\def\dar[#1]{\ar@<2pt>[#1]\ar@<-2pt>[#1]}

\title{Hodge conjecture for projective hypersurfaces}

\author{Johann Bouali}

\begin{document}

\maketitle

\begin{abstract}
We show that a Hodge class of a complex smooth projective hypersurface is an analytic logarithmic De Rham class.
On the other hand we show that for a complex smooth projective variety 
an analytic logarithmic De Rham class of type $(d,d)$ is the class of a codimension $d$ algebraic cycle. 
We deduce the Hodge conjecture for smooth projective hypersurfaces.
\end{abstract}

\tableofcontents

\section{Introduction}

\subsection{Complex analytic De Rham logarithmic classes}

In \cite{B7}, we introduced, for $X$ an algebraic variety over a field,
the notion of logarithmic De Rham cohomology classes.
For $X$ a complex algebraic variety, we also introduce in this work the notion of 
logarithmic analytic De Rham cohomology classes. 
The notion of log forms was introduced by Bloch and Illusie 
in the context of the De Rham-Witt complex of algebraic varieties in positive characteristics. 
It is different from the notion of forms with logarithmic poles as introduced by Deligne in Hodge II theory.
In proposition \ref{cGAGAlog} (section 4)
we give an analytic analogue of \cite{B7} theorem 2 (ii), that is if $X$ is a smooth complex projective variety of dimension $d_X$, 
\begin{itemize}
\item a logarithmic analytic De Rham cohomology class of type $(d,d)$ is the class of a codimension $d$ algebraic cycle.
\item we have $H^jOL_{X^{an},0}(H_{usu}^{j-l}(X^{an},\Omega^l_{X^{an},\log,0}))=0$ for $j,l\in\mathbb Z$ such that $2l>j$.
\end{itemize}
Here 
\begin{equation*}
OL_{X^{an},0}:\Omega^l_{X^{an},\log,0}[-l]:=(\wedge^lO_{X^{an}}^*)\otimes\mathbb Q[-l]\to\Omega^\bullet_{X^{an}} 
\end{equation*}
is the morphism of the abstract abelian group generated by logarithmic analytic $l$-forms (see definition \ref{omegazero})
to the analytic de Rham complex and $usu$ denote the usual complex topology.  
The proof of proposition \ref{cGAGAlog} is motivic and similar to the proof of \cite{B7} theorem 2 :
a logarithmic analytic class of bidegree $(p,q)$, $p\geq q$
which comes from $X$ is acyclic for the Zariski topology on each affine open subset $U\subset X$ such that 
$\Pic(U)=0$ and $D:=X\backslash U=D'\cup E$ 
where $D'\subset X$ is an ample divisor and $E\subset X$ is a smooth ample divisor (proposition \ref{UXusu}), 
this allows us to proceed by a finite induction 
using the crucial fact that the purity isomorphism for De Rham cohomology preserve logarithmic analytic classes
since the purity isomorphism is motivic (c.f. \cite{CD}, see proposition \ref{smXZ}).

\subsection{Hodge conjecture for smooth projective hypersurfaces}

Let $X=V(f)\subset\mathbb P^N_{\mathbb C}$ be a smooth projective hypersurface. 
Denote $j:U:=\mathbb P^N_{\mathbb C}\backslash X\hookrightarrow\mathbb P^N_{\mathbb C}$ the open complementary subset.
Let  $(\Omega^{\bullet}_{\mathbb P^N}\otimes_{O_{\mathbb P^N}}F^{\bullet-p}j_*O_U)^{an}$ 
be the filtered De Rham complex of $U$. The cone of $\iota_{\partial}$ : 
\begin{equation*}
\iota_{\partial}:\Omega^{N-p}_{\mathbb P^{N,an}}(\log X)^{\partial=0}[-p]\hookrightarrow
(\Omega^{\bullet}_{\mathbb P^N}\otimes_{O_{\mathbb P^N}}F^{\bullet-p}j_*O_U)^{an}
\end{equation*}
is locally acyclic for the complex topology, more precisely is acyclic on open balls (see e.g.\cite{Voisin}). 
As a consequence, we show in proposition \ref{HUQ}, 
using an open cover $\mathbb P^N_{\mathbb C}=\cup_{i=1}^n B_i$ by open balls $\phi_i:B_i\simeq D(0,1)^N$
such that $\phi_i(B_i\cap X)=\left\{0\right\}\times D(0,1)^{N-1}$, 
that a Hodge class $\alpha\in F^{N-p}H^N(U^{an},\mathbb Q)$ is logarithmic, that is 
\begin{equation*}
\alpha\in H^NOL_{U^{an},0}(H_{usu}^p(U^{an},\Omega_{U^{an},\log,0}^{N-p})).
\end{equation*}
We get that, if $N=2p+1$ and $\lambda\in F^pH^{2p}(X^{an},\mathbb Q)$ is a Hodge class,  
\begin{equation*}
\lambda=Res_{X,\mathbb P^N}(\alpha)+m[H]\in H^{2p}OL_{X^{an},0}(H_{usu}^p(X^{an},\Omega_{X^{an},\log,0}^p)),
\end{equation*}
where $Res_{X,\mathbb P^N}:H^N(U^{an},\mathbb C)\xrightarrow{\sim}H^{2p}(X^{an},\mathbb C)$
is the residu map, using the fact that since the residu map is motivic (see definition \ref{ResMot}), 
it preserves the logaritmic De Rham classes. 
Proposition \ref{cGAGAlog} (ii) then implies that $\lambda=[Z]$, $Z\in\mathcal Z^p(X)$, is the class of an algebraic cycle.

To our knowledge, Hodge's conjecture for hypersurfaces was known for only a few cases : hyperplanes, 
quadrics and Fermat's hypersurfaces of degree $\leq 21$ (Shioda). 
Our new contribution which allowed us to tackle the general case of 
hypersurfaces is the introduction of the notion of analytical logarithmic de Rham classes on the one hand 
and the motivic purity isomorphism on the other hand.

I am grateful to Professor F. Mokrane for his help and support during the preparation of this work.

\section{Preliminaries and Notations}

\subsection{Notations}

\begin{itemize}

\item Denote by $\Top$ the category of topological spaces and $\RTop$ the category of ringed spaces.
\item Denote by $\Cat$ the category of small categories and $\RCat$ the category of ringed topos.
\item For $\mathcal S\in\Cat$ and $X\in\mathcal S$, we denote $\mathcal S/X\in\Cat$ the category whose
objects are $Y/X:=(Y,f)$ with $Y\in\mathcal S$ and $f:Y\to X$ is a morphism in $\mathcal S$, and whose morphisms
$\Hom((Y',f'),(Y,f))$ consists of $g:Y'\to Y$ in $\mathcal S$ such that $f\circ g=f'$.

\item For $(\mathcal S,O_S)\in\RCat$ a ringed topos, we denote by 
\begin{itemize}
\item $\PSh(\mathcal S)$ the category of presheaves of $O_S$ modules on $\mathcal S$ and
$\PSh_{O_S}(\mathcal S)$ the category of presheaves of $O_S$ modules on $\mathcal S$, 
whose objects are $\PSh_{O_S}(\mathcal S)^0:=\left\{(M,m),M\in\PSh(\mathcal S),m:M\otimes O_S\to M\right\}$,
together with the forgetful functor $o:\PSh(\mathcal S)\to \PSh_{O_S}(\mathcal S)$,
for $F\in\PSh(\mathcal S)$ and $X\in\mathcal S$ we denote $F(X):=\Gamma(X,F)$ its sections over $X$,
\item $C(\mathcal S)=C(\PSh(\mathcal S))$ and $C_{O_S}(\mathcal S)=C(\PSh_{O_S}(\mathcal S))$ 
the big abelian category of complexes of presheaves of $O_S$ modules on $\mathcal S$,
\item $C_{O_S(2)fil}(\mathcal S):=C_{(2)fil}(\PSh_{O_S}(\mathcal S))\subset C(\PSh_{O_S}(\mathcal S),F,W)$,
the big abelian category of (bi)filtered complexes of presheaves of $O_S$ modules on $\mathcal S$ 
such that the filtration is biregular and $\PSh_{O_S(2)fil}(\mathcal S):=(\PSh_{O_S}(\mathcal S),F,W)$.
\end{itemize}

\item Let $(\mathcal S,O_S)\in\RCat$ a ringed topos with topology $\tau$. For $F\in C_{O_S}(\mathcal S)$,
we denote by $k:F\to E_{\tau}(F)$ the canonical flasque resolution in $C_{O_S}(\mathcal S)$ (see \cite{B5}).
In particular for $X\in\mathcal S$, $H^*(X,E_{\tau}(F))\xrightarrow{\sim}\mathbb H_{\tau}^*(X,F)$.

\item For $f:\mathcal S'\to\mathcal S$ a morphism with $\mathcal S,\mathcal S'\in\RCat$,
endowed with topology $\tau$ and $\tau'$ respectively, we denote for $F\in C_{O_S}(\mathcal S)$ and each $j\in\mathbb Z$,
\begin{itemize}
\item $f^*:=H^j\Gamma(\mathcal S,k\circ\ad(f^*,f_*)(F)):\mathbb H^j(\mathcal S,F)\to\mathbb H^j(\mathcal S',f^*F)$,
\item $f^*:=H^j\Gamma(\mathcal S,k\circ\ad(f^{*mod},f_*)(F)):\mathbb H^j(\mathcal S,F)\to\mathbb H^j(\mathcal S',f^{*mod}F)$, 
\end{itemize}
the canonical maps.

\item For $m:A\to B$, $A,B\in C(\mathcal A)$, $\mathcal A$ an additive category, 
we denote $c(A):\Cone(m:A\to B)\to A[1]$ and $c(B):B\to\Cone(m:A\to B)$ the canonical maps.

\item For $X\in\Top$ and $Z\subset X$ a closed subset, denoting $j:X\backslash Z\hookrightarrow X$ the open complementary, 
we will consider
\begin{equation*}
\Gamma^{\vee}_Z\mathbb Z_X:=\Cone(\ad(j_!,j^*)(\mathbb Z_X):j_!j^*\mathbb Z_X\hookrightarrow\mathbb Z_X)\in C(X)
\end{equation*}
and denote for short $\gamma^{\vee}_Z:=\gamma^{\vee}_X(\mathbb Z_X):=c(\mathbb Z_X):\mathbb Z_X\to\Gamma^{\vee}_Z\mathbb Z_X$
the canonical map in $C(X)$.

\item Denote by $\Sch\subset\RTop$ the subcategory of schemes (the morphisms are the morphisms of locally ringed spaces).
We denote by $\PSch\subset\Sch$ the full subcategory of proper schemes.
For a field $k$, we consider $\Sch/k:=\Sch/\Spec k$ the category of schemes over $\Spec k$.
The objects are $X:=(X,a_X)$ with $X\in\Sch$ and $a_X:X\to\Spec k$ a morphism
and the morphisms are the morphisms of schemes $f:X'\to X$ such that $f\circ a_{X'}=a_X$. We then denote by
\begin{itemize}
\item $\Var(k)=\Sch^{ft}/k\subset\Sch/k$ the full subcategory consisting of algebraic varieties over $k$, 
i.e. schemes of finite type over $k$,
\item $\PVar(k)\subset\QPVar(k)\subset\Var(k)$ 
the full subcategories consisting of quasi-projective varieties and projective varieties respectively, 
\item $\PSmVar(k)\subset\SmVar(k)\subset\Var(k)$,  $\PSmVar(k):=\PVar(k)\cap\SmVar(k)$,
the full subcategories consisting of smooth varieties and smooth projective varieties respectively.
\end{itemize}

\item For $X\in\Sch$ a noetherian scheme and $p\in\mathbb N$, we denote $\mathcal Z^p(X)$ 
the $\mathbb Q$-vector space generated by closed subset of codimension $p$.

\item Denote by $\AnSp(\mathbb C)\subset\RTop$ the subcategory of analytic spaces over $\mathbb C$,
and by $\AnSm(\mathbb C)\subset\AnSp(\mathbb C)$ the full subcategory of smooth analytic spaces 
(i.e. complex analytic manifold).

\item For $X\in\Top$ and $X=\cup_{i\in I}X_i$ an open cover, 
we denote $X_{\bullet}\in\Fun(\Delta,\Top)$ the associated simplicial space,
with for $J\subset I$, $j_{IJ}:X_I:=\cap_{i\in I}X_i\hookrightarrow X_J:=\cap_{i\in J}X_i$ the open embedding. 

\item For $X\in\AnSm(\mathbb C)$ and $X=\cup_{i\in I}X_i$ an open cover, 
we denote $X_{\bullet}\in\Fun(\Delta,\AnSm(\mathbb C))$ the associated simplicial complex manifold,
with for $J\subset I$, $j_{IJ}:X_I:=\cap_{i\in I}X_i\hookrightarrow X_J:=\cap_{i\in J}X_i$ the open embedding. 

\item Let $(X,O_X)\in\RTop$. We consider its De Rham complex $\Omega_X^{\bullet}:=DR(X)(O_X)$.
\begin{itemize}
\item Let $X\in\Sch$. Considering its De Rham complex $\Omega_X^{\bullet}:=DR(X)(O_X)$,
we have for $j\in\mathbb Z$ its De Rham cohomology $H^j_{DR}(X):=\mathbb H^j(X,\Omega^{\bullet}_X)$.
\item Let $X\in\Var(k)$. Considering its De Rham complex $\Omega_X^{\bullet}:=\Omega_{X/k}^{\bullet}:=DR(X/k)(O_X)$,
we have for $j\in\mathbb Z$ its De Rham cohomology $H^j_{DR}(X):=\mathbb H^j(X,\Omega^{\bullet}_X)$.
The differentials of $\Omega_X^{\bullet}:=\Omega_{X/k}^{\bullet}$ are by definition $k$-linear,
thus $H^j_{DR}(X):=\mathbb H^j(X,\Omega^{\bullet}_X)$ has a structure of a $k$ vector space.
\item Let $X\in\AnSp(\mathbb C)$. Considering its De Rham complex $\Omega_X^{\bullet}:=DR(X)(O_X)$,
we have for $j\in\mathbb Z$ its De Rham cohomology $H^j_{DR}(X):=\mathbb H^j(X,\Omega^{\bullet}_X)$.
\end{itemize}

\item For $X\in\AnSp(\mathbb C)$, we denote $\alpha(X):\mathbb C_X\hookrightarrow\Omega_X^{\bullet}$ the embedding in $C(X)$.
For $X\in\AnSm(\mathbb C)$, $\alpha(X):\mathbb C_X\hookrightarrow\Omega_X^{\bullet}$ 
is an equivalence usu local by Poincare lemma.

\item We denote $\mathbb I^n:=[0,1]^n\in\Diff(\mathbb R)$ (with boundary).
For $X\in\Top$ and $R$ a ring, we consider its singular cochain complex
\begin{equation*}
C^*_{\sing}(X,R):=(\mathbb Z\Hom_{\Top}(\mathbb I^*,X)^{\vee})\otimes R 
\end{equation*}
and for $l\in\mathbb Z$ its singular cohomology $H^l_{\sing}(X,R):=H^nC^*_{\sing}(X,R)$.
For $f:X'\to X$ a continuous map with $X,X'\in\Top$, we have the canonical map of complexes
\begin{equation*}
f^*:C^*_{\sing}(X,R)\to C^*_{\sing}(X,R), \sigma\mapsto f^*\sigma:=(\gamma\mapsto\sigma(f\circ\gamma)).
\end{equation*}
In particular, we get by functoriality the complex 
\begin{equation*}
C^*_{X,R\sing}\in C_R(X), \; (U\subset X)\mapsto C^*_{\sing}(U,R)
\end{equation*}
We recall that 
\begin{itemize}
\item For $X\in\Top$ locally contractible, e.g. $X\in\CW$, and $R$ a ring, the inclusion in $C_R(X)$
$c_X:R_X\to C^*_{X,R\sing}$ is by definition an equivalence top local and that we get 
by the small chain theorem, for all $l\in\mathbb Z$, an isomorphism 
$H^lc_X:H^l(X,R_X)\xrightarrow{\sim}H^l_{\sing}(X,R)$.
\item For $X\in\Diff(\mathbb R)$, the restriction map 
\begin{equation*}
r_X:\mathbb Z\Hom_{\Diff(\mathbb R)}(\mathbb I^*,X)^{\vee}\to C^*_{\sing}(X,R), \; 
w\mapsto w:(\phi\mapsto w(\phi))
\end{equation*}
is a quasi-isomorphism by Whitney approximation theorem.
\end{itemize}

\item We will consider the morphism of site given by the analytification functor
\begin{equation*}
\An:\AnSp(\mathbb C)\to\Var(\mathbb C), \; \; X\mapsto\An(X):=X^{an}, \; (f:X'\to X)\mapsto\An(f):=f^{an}.
\end{equation*}
For $X\in\Var(\mathbb C)$, we denote by 
\begin{equation*}
\an_X:=\An_{|X^{et}}:=X^{an,et}\to X^{et}, \; \an_X(U\to X):=U^{an}, 
\end{equation*}
its restriction to the small etale site $o_X(X^{et})\subset\Var(\mathbb C)$, where $o_X(U\to X)=U$,
in particular we have the following commutative diagram
\begin{equation*}
\xymatrix{\AnSp(\mathbb C)\ar[r]^{\An}\ar[d]^{o_{X^{an}}} & \Var(\mathbb C)\ar[d]^{o_X} \\
X^{an,et}\ar[r]^{\an_X} & X^{et}},
\end{equation*}
where $o_{X^{an}}(U\hookrightarrow X^{an})=U$. 
We get the morphism in $\RTop$ $\an_X:=\an_{X|O(X)}:X^{an}\to X$, 
$o:O(X)\subset X^{et}$ being the embedding of the Zariski open subsets of $X$.
\end{itemize}

\subsection{Complex integral periods}

Let $k$ be a field of characteristic zero and $X\in\SmVar(k)$ be a smooth variety. 
Let $X=\cup_{i=1}^sX_i$ be an open affine cover. 
For each embedding $\sigma:k\hookrightarrow\mathbb C$, we have
the evaluation period embedding map which is the morphism of bi-complexes 
\begin{eqnarray*}
ev(X)^{\bullet}_{\bullet}:\Gamma(X_{\bullet},\Omega^{\bullet}_{X_{\bullet}})\to 
\mathbb Z\Hom_{\Diff}(\mathbb I^{\bullet},X^{an}_{\mathbb C,\bullet})^{\vee}\otimes\mathbb C, \\
w^l_I\in\Gamma(X_I,\Omega^l_{X_I})\mapsto 
(ev(X)^l_I(w^l_I):\phi^l_I\in\mathbb Z\Hom_{\Diff}(\mathbb I^l,X^{an}_{\mathbb C,I})^{\vee}\otimes\mathbb C
\mapsto ev^l_I(w^l_I)(\phi^l_I):=\int_{\mathbb I^l}\phi_I^{l*}w^l_I)
\end{eqnarray*}
given by integration. By taking all the affine open cover $(j_i:X_i\hookrightarrow X)$ of $X$,
we get for $\sigma:k\hookrightarrow\mathbb C$, the evaluation period embedding map 
\begin{eqnarray*}
ev(X):=\varinjlim_{(j_i:X_i\hookrightarrow X)}ev(X)^{\bullet}_{\bullet}:
\varinjlim_{(j_i:X_i\hookrightarrow X)}\Gamma(X_{\bullet},\Omega^{\bullet}_{X_{\bullet}})
\to \varinjlim_{(j_i:X_i\hookrightarrow X)}
\mathbb Z\Hom_{\Diff(\mathbb R)}(\mathbb I^{\bullet},X^{an}_{\mathbb C,\bullet})^{\vee}\otimes\mathbb C
\end{eqnarray*}
It induces in cohomology, for $j\in\mathbb Z$, the evaluation period map
\begin{eqnarray*}
H^jev(X)=H^jev(X)^{\bullet}_{\bullet}:H^j_{DR}(X)=H^j\Gamma(X_{\bullet},\Omega^{\bullet}_{X_{\bullet}})\to 
H_{\sing}^j(X_{\mathbb C}^{an},\mathbb C)=
H^j(\Hom_{\Diff(\mathbb R)}(\mathbb I^{\bullet},X^{an}_{\mathbb C,\bullet})^{\vee}\otimes\mathbb C). 
\end{eqnarray*}
which does NOT depend on the choice of the affine open cover 
by acyclicity of quasi-coherent sheaves on affine noetherian schemes for the left hand side
and from Mayer-Vietoris quasi-isomorphism for singular cohomology of topological spaces
and Whitney approximation theorem for differential manifolds for the right hand side.

Let $X\in\SmVar(k)$. Then, 
\begin{equation*}
H^*ev(X_{\mathbb C})=H^*R\Gamma(X_{\mathbb C}^{an},\alpha(X))^{-1}\circ\Gamma(X_{\mathbb C}^{an},E_{zar}(\Omega(\an_X))):
H^*_{DR}(X_{\mathbb C})\xrightarrow{\sim}H^*_{\sing}(X_{\mathbb C}^{an},\mathbb C)
\end{equation*}
is the canonical isomorphism induced by the analytical functor and the quasi-isomorphism 
$\alpha(X):\mathbb C_{X_{\mathbb C}^{an}}\hookrightarrow\Omega^{\bullet}_{X^{an}_{\mathbb C}}$, 
which gives the periods elements $H^*ev(X)(H^*_{DR}(X))\subset H^*_{\sing}(X_{\mathbb C}^{an},\mathbb C)$.

\subsection{Algebraic cycles and motives}

For $X\in\Sch$ noetherian irreducible and $d\in\mathbb N$, 
we denote by $\mathcal Z^d(X)$ the group of algebraic cycles of codimension $d$,
which is the free abelian group generated by irreducible closed subsets of codimension $d$.

For $X,X'\in\Sch$ noetherian, with $X'$ irreducible, we denote by 
$\mathcal Z^{fs/X'}(X'\times X)\subset\mathcal Z_{d_{X'}}(X'\times X)$
which consist of algebraic cycles $\alpha=\sum_in_i\alpha_i\in\mathcal Z_{d_{X'}}(X'\times X)$ such that,
denoting $\supp(\alpha)=\cup_i\alpha_i\subset X'\times X$ its support and $p':X'\times X\to X'$ the projection,
$p'_{|\supp(\alpha)}:\supp(\alpha)\to X'$ is finite surjective.

Let $k$ be a field.
We denote by $\Cor\SmVar(k)$ the category whose objects are $\left\{X\in\SmVar(k)\right\}$ and 
\begin{equation*}
\Hom_{\Cor\SmVar(k)}(X',X):=\mathcal Z^{fs/X'}(X'\times X) 
\end{equation*}
(see \cite{CD} for the composition law).
We denote by $\Tr:\Cor\SmVar(k)\to\SmVar(k)$ the morphism of site given by the embedding 
$\Tr:\SmVar(k)\hookrightarrow\Cor\SmVar(k)$. Let $F\in\PSh(\SmVar(k))$, we say that $F$ admits
transfers if $F=\Tr_*F$ with $F\in\PSh(\Cor\SmVar(k))$. 

Let $k$ be a field.
We recall that $\DA(k):=\Ho_{\mathbb A^1,et}(C(\SmVar(k)))$ is the derived category of mixed motives.

We have (see e.g. \cite{B5}) the Hodge realization functor
\begin{equation*}
\mathcal F^{Hdg}:\DA(\mathbb C)\to D(MHS), \: 
M\mapsto\mathcal F^{Hdg}(M):=\iota^{-1}(\mathcal F^{FDR}(M),\Bti^*M,\alpha(M)),
\end{equation*} 
where $D(MHS)\hookrightarrow D_{2fil}(\mathbb C)\times_ID_{fil}(\mathbb Q)$ 
is the derived category of mixed Hodge structures.

For $X\in\Var(k)$ and $Z\subset X$ a closed subset, denoting $j:X\backslash Z\hookrightarrow X$ the open complementary, 
we will consider
\begin{equation*}
\Gamma^{\vee}_Z\mathbb Z_X:=\Cone(\ad(j_{\sharp},j^*)(\mathbb Z_X):
j_{\sharp}j^*\mathbb Z_X\hookrightarrow\mathbb Z_X)\in C(\Var(k)^{sm}/X)
\end{equation*}
and denote for short 
\begin{equation*}
\gamma^{\vee}_Z:=\gamma^{\vee}_Z(\mathbb Z_X):=c(\mathbb Z_X):
\mathbb Z_X\to\Gamma^{\vee}_Z\mathbb Z_X
\end{equation*}
the canonical map in $C(\Var(k)^{sm}/X)$.
For $X\in\Var(k)$ and $Z\subset X$ a closed subset, denoting $a_X:X\to\Spec k$ the structural map,
we have then the motive of $X$ with support in $Z$ defined as 
\begin{equation*}
M_Z(X):=a_{X!}\Gamma^{\vee}_Za_X^!\mathbb Z\in\DA(k).
\end{equation*}
If $X\in\SmVar(k)$, we will also consider
\begin{equation*}
a_{X\sharp}\Gamma^{\vee}_Z\mathbb Z_X:=\Cone(a_{X\sharp}\circ\ad(j_{\sharp},j^*)(\mathbb Z_X):
\mathbb Z(U)\hookrightarrow\mathbb Z(X))=:\mathbb Z(X,X\backslash Z)\in C(\SmVar(k)),
\end{equation*}
and denote for short 
\begin{equation*}
\gamma^{\vee}_Z:=a_{X\sharp}\gamma^{\vee}_Z(\mathbb Z_X):=c(\mathbb Z(X)):\mathbb Z(X)\to\mathbb Z(X,X\backslash Z)
\end{equation*}
the map in $C(\SmVar(k))$.
Then for $X\in\SmVar(k)$ and $Z\subset X$ a closed subset
\begin{equation*}
M_Z(X):=a_{X!}\Gamma^{\vee}_Za_X^!\mathbb Z
=a_{X\sharp}\Gamma^{\vee}_Z\mathbb Z_X=:\mathbb Z(X,X\backslash Z)\in\DA(k).
\end{equation*}

\begin{itemize}
\item Let $(X,Z)\in\Sch^2$ with $X\in\Sch$ a noetherian scheme and $Z\subset X$ a closed subset.
We have the deformation $(D_ZX,\mathbb A^1_Z)\to\mathbb A^1$, $(D_ZX,\mathbb A^1_Z)\in\Sch^2$ 
of $(X,Z)$ by the normal cone $C_{Z/X}\to Z$, i.e. such that 
\begin{equation*}
(D_ZX,\mathbb A^1_Z)_s=(X,Z), \, s\in\mathbb A^1\backslash 0, \;  (D_ZX,\mathbb A^1_Z)_0=(C_{Z/X},Z).
\end{equation*}
We denote by $i_1:(X,Z)\hookrightarrow (D_ZX,\mathbb A^1_Z)$ and 
$i_0:(C_{Z/X},Z)\hookrightarrow (D_ZX,\mathbb A^1_Z)$ the closed embeddings in $\Sch^2$.
\item Let $k$ be a field of characteristic zero. Let $X\in\SmVar(k)$. 
For $Z\subset X$ a closed subset of pure codimension $c$,
consider a desingularisation $\epsilon:\tilde Z\to Z$ of $Z$ and denote $n:\tilde Z\xrightarrow{\epsilon}Z\subset X$.
We have then the morphism in $\DA(k)$
\begin{equation*}
G_{Z,X}:M(X)\xrightarrow{D(\mathbb Z(n))}M(\tilde Z)(c)[2c]\xrightarrow{\mathbb Z(\epsilon)}M(Z)(c)[2c]
\end{equation*}
where $D:\Hom_{\DA(k)}(M_c(\tilde Z),M_c(X))\xrightarrow{\sim}\Hom_{\DA(k)}(M(X),M(\tilde Z)(c)[2c])$
is the duality isomorphism from the six functors formalism (moving lemma of Suzlin and Voevodsky)
and $\mathbb Z(n):=\ad(n_!,n^!)(a_X^!\mathbb Z)$, noting that $n_!=n_*$ since $n$ is proper and that
$a_X^!=a_X^*[d_X]$ and $a_{\tilde Z}^!=a_{\tilde Z}^*[d_Z]$ since $X$, resp. $\tilde Z$, are smooth
(considering the connected components, we may assume $X$ and $\tilde Z$ of pure dimension).
\end{itemize}

We recall the following facts (see \cite{CD} and \cite{B5}):

\begin{prop}\label{smXZ}
Let $k$ be a field a characteristic zero.
Let $X\in\SmVar(k)$ and $i:Z\subset X$ a smooth subvariety of pure codimension d.
Then $C_{Z/X}=N_{Z/X}\to Z$ is a vector bundle of rank $d$.
The closed embeddings $i_1:(X,Z)\hookrightarrow (D_ZX,\mathbb A^1_Z)$ and 
$i_0:(C_{Z/X},Z)\hookrightarrow (D_ZX,\mathbb A^1_Z)$ in $\SmVar^2(k)$ induces isomorphisms of motives
$\mathbb Z(i_1):M_Z(X)\xrightarrow{\sim}M_{\mathbb A^1_Z}(D_ZX)$ and 
$\mathbb Z(i_0):M_Z(N_{Z/X})\xrightarrow{\sim}M_{\mathbb A^1_Z}(D_ZX)$ in $\DA(k)$.
We get the excision isomorphism in $\DA(k)$
\begin{equation*}
P_{Z,X}:=\mathbb Z(i_0)^{-1}\circ\mathbb Z(i_1):M_Z(X)\xrightarrow{\sim}M_Z(N_{Z/X}).
\end{equation*}
We have 
\begin{equation*}
Th(N_{Z/X})\circ P_{Z,X}\circ\gamma^{\vee}_Z(\mathbb Z_X)=G_{Z,X}:=D(\mathbb Z(i)):M(X)\to M(Z)(d)[2d].
\end{equation*}
\end{prop}

\begin{proof}
See \cite{CD}.
\end{proof}

\begin{defi}\label{ResMot}
Let $k$ be a field a characteristic zero.
Let $X\in\SmVar(k)$ and $i:Z\subset X$ a smooth subvariety of pure codimension d.
We consider using proposition \ref{smXZ} the canonical map in $\DA(k)$
\begin{equation*}
Res_{Z,X}:M(Z)(d)[-2d]\xrightarrow{P_{Z,X}^{-1}\circ Th(N_{Z/X})^{-1}}M_Z(X)
\xrightarrow{c(\mathbb Z(X\backslash Z))}M(X\backslash Z)[1].
\end{equation*}
We recall that $c(\mathbb Z(X\backslash Z)):\mathbb Z(X,X\backslash Z)\to\mathbb Z(X\backslash Z)[1]$ 
is the canonical map in $C(\SmVar(k))$.
\end{defi}

\begin{itemize}
\item Recall that we say that $F\in\PSh(\SmVar(k))$ is $\mathbb A^1$ invariant if for all $X\in\SmVar(k)$,
$p^*:=F(p):F(X)\to F(X\times\mathbb A^1)$ is an isomorphism, where $p:X\times\mathbb A^1\to X$ is the projection.
\item Similarly, we say that $F\in\PSh(\AnSm(\mathbb C))$ is $\mathbb A^1$ invariant if for all $X\in\AnSm(\mathbb C)$,
$p^*:=F(p):F(X)\to F(X\times\mathbb A^1)$ is an isomorphism, where $p:X\times\mathbb A^1\to X$ is the projection.
\end{itemize}

\subsection{The logaritmic De Rham complexes}

We recall from \cite{B7} the following notion :

\begin{itemize}
\item For $k$ a field and $X\in\Var(k)$, we consider the embedding in $C(X)$
\begin{equation*}
OL_X:\Omega^{\bullet}_{X,\log}\hookrightarrow\Omega_X^{\bullet}:=\Omega^{\bullet}_{X/k},
\end{equation*}
such that, for $X^o\subset X$ an open subset and $w\in\Omega^p_X(X^o)$,
$w\in\Omega^p_{X,\log}(X^o)$ if and only if there exists $(n_i)_{1\leq i\leq s}\in\mathbb Z$ and 
$(f_{i,\alpha_k})_{1\leq i\leq s,1\leq k\leq p}\in\Gamma(X^o,O_X)^*$ such that
\begin{equation*}
w=\sum_{1\leq i\leq s}n_idf_{i,\alpha_1}/f_{i,\alpha_1}\wedge\cdots\wedge df_{i,\alpha_p}/f_{i,\alpha_p}
\in\Omega^p_X(X^o),
\end{equation*}
and for $p=0$, $\Omega^0_{X,\log}:=\mathbb Z$. Let $k$ be a field. We get an embedding in $C(\Var(k))$
\begin{eqnarray*}
OL:\Omega^{\bullet}_{/k,\log}\hookrightarrow\Omega^{\bullet}_{/k}, \; 
\mbox{given by}, \, \mbox{for} \, X\in\Var(k), \\
OL(X):=OL_X:\Omega^{\bullet}_{/k,\log}(X):=\Gamma(X,\Omega^{\bullet}_{X,\log})
\hookrightarrow\Gamma(X,\Omega^{\bullet}_X)=:\Omega^{\bullet}_{/k}(X)
\end{eqnarray*}
and its restriction to $\SmVar(k)\subset\Var(k)$.
\item For $X\in\AnSp(\mathbb C)$, we consider the embedding in $C(X)$
\begin{equation*}
OL_X:\Omega^{\bullet}_{X,\log}\hookrightarrow\Omega_X^{\bullet}:=\Omega^{\bullet}_{X/\mathbb C},
\end{equation*}
such that, for $X^o\subset X$ an open subset and $w\in\Omega^p_X(X^o)$,
$w\in\Omega^p_{X,\log}(X^o)$ if and only if there exists $(n_i)_{1\leq i\leq s}\in\mathbb Z$ and 
$(f_{i,\alpha_k})_{1\leq i\leq s,1\leq k\leq p}\in\Gamma(X^o,O_X)^*$ such that
\begin{equation*}
w=\sum_{1\leq i\leq s}n_idf_{i,\alpha_1}/f_{i,\alpha_1}\wedge\cdots\wedge df_{i,\alpha_p}/f_{i,\alpha_p}
\in\Omega^p_X(X^o),
\end{equation*}
and for $p=0$, $\Omega^0_{X,\log}:=\mathbb Z$. The differential is by definition $\mathbb C$ linear.
Recall that for $f\in O(X)$, $df=0$ if and only if $f$ is locally constant with complex values, that is
$f_{|X_i^o}=\lambda_i\in\mathbb C$ for each $1\leq i\leq r$, 
where $X=\sqcup_{i=1}^r X_i$ is the decomposition into connected components.
We get an embedding in $C(\AnSp(\mathbb C))$
\begin{eqnarray*}
OL^{an}:\Omega^{\bullet,an}_{\log}\hookrightarrow\Omega^{\bullet,an}, \; 
\mbox{given by}, \, \mbox{for} \, X\in\AnSp(\mathbb C), \\
OL^{an}(X):=OL_X:\Omega^{\bullet,an}_{\log}(X):=\Gamma(X,\Omega^{\bullet}_{X,\log})
\hookrightarrow\Gamma(X,\Omega^{\bullet}_X)=:\Omega^{\bullet,an}(X)
\end{eqnarray*}
and its restriction to $\AnSm(\mathbb C)\subset\AnSp(\mathbb C)$.
\end{itemize}

In the analytic case we also consider the following
\begin{defi}\label{omegazero}
For $X\in\AnSp(\mathbb C)$, we consider the morphism in $C(X)$
\begin{equation*}
OL_{X,0}:\Omega^{\bullet}_{X,\log,0}:=(\wedge^{\bullet} O_X^*)\otimes\mathbb Q\xrightarrow{q_L}\Omega^{\bullet}_{X,\log}
\xrightarrow{OL_X}\Omega_X^{\bullet},
\end{equation*}
where, for $X^o\subset X$ an open subset, $p\in\mathbb N$,
\begin{eqnarray*}
\Omega^{p}_{X,\log,0}(X^o):=\wedge^pO_X^*(X^o)\otimes\mathbb Q =
<df_{\alpha_1}/f_{\alpha_1}\wedge\cdots\wedge df_{\alpha_p}/f_{\alpha_p}, f_{\alpha_j}\in O(X^o)^*>/ \\
<df_{\alpha_1}/f_{\alpha_1}\wedge\cdots\wedge d(f_{\alpha_j}f'_{\alpha_j})/(f_{\alpha_j}f'_{\alpha_j})
\wedge\cdots\wedge df_{\alpha_p}/f_{\alpha_p} \\
-df_{\alpha_1}/f_{\alpha_1}\wedge\cdots\wedge df_{\alpha_j}/f_{\alpha_j}\wedge\cdots\wedge df_{\alpha_p}/f_{\alpha_p}
-df_{\alpha_1}/f_{\alpha_1}\wedge\cdots\wedge df'_{\alpha_j}/f'_{\alpha_j}\wedge\cdots\wedge df_{\alpha_p}/f_{\alpha_p}>\otimes\mathbb Q 
\end{eqnarray*}
is the abstract abelian group generated by logarithmic forms,
for $p<0$ of course $\Omega^p_{X,\log,0}=0$, and by definition all the differentials of this complex are trivial.
Note that $\Omega^1_{X,\log,0}=O_X^*\otimes\mathbb Q $.
For $f\in O(X^o)^*$, $df=0\in\Omega^1_{X,\log,0}(X^o)$ if and only if $f=1$. 
Also note that the sheaf of abelian groups $O_X^*$ is NOT endowed with a structure of $O_X$ module and that we take the wedge product of sheaves of abelian groups.
In particular $OL_{X,0}$ is NOT injective.  
For $\lambda\in\mathbb C$ and $f\in O(X^o)^*$, 
\begin{equation*}
d(\lambda f)/(\lambda f)\neq df/f\in\Omega^1_{X,\log,0}(X^o), \;
d(\lambda f)/(\lambda f)-df/f=d\lambda/\lambda\in\Omega^1_{X,\log,0}(X^o).
\end{equation*}
We get the morphism in $C(\AnSp(\mathbb C))$
\begin{eqnarray*}
OL_0^{an}:\Omega^{\bullet,an}_{\log,0}\xrightarrow{q_L}\Omega^{\bullet,an}_{\log}\xrightarrow{OL^{an}}\Omega^{\bullet,an}, \; 
\mbox{given by}, \, \mbox{for} \, X\in\AnSp(\mathbb C), \\
OL_0^{an}(X):=OL_{X,0}:\Omega^{\bullet,an}_{\log,0}(X):=\Gamma(X,\Omega^{\bullet}_{X,\log,0})
\to\Gamma(X,\Omega^{\bullet}_X)=:\Omega^{\bullet,an}(X)
\end{eqnarray*}
and its restriction to $\AnSm(\mathbb C)\subset\AnSp(\mathbb C)$.
\end{defi}

\begin{rem}\label{ballclog}
Note that for $X\in\AnSp(\mathbb C)$
and $\partial f\in\Omega^{1,an}(X^{an})$ with $f\in O(X^{an})$, we have 
$\partial f=(\partial e^f)/e^f\in\Omega^{1,an}_{\log}(X^{an})$, 
hence, since open balls form a basis of the complex topology on complex analytic manifold, 
for each $p\in\mathbb Z$, $p\neq 0$, 
\begin{equation*}
a_{usu}(\Omega^{p,an}_{\log})=\Omega^{p,an,\partial=0}=a_{usu}\partial\Omega^{p-1,an}.
\end{equation*}
On the other hand, for $X\in\AnSp(\mathbb C)$, $H^1(X,\Omega^{1,an}_{\log,0})=\Pic(X)$ is the Picard group,
$H^1(X,\Omega^{1,an,\partial=0})=\mathbb H^2(X,F^1\Omega_X^{\bullet})$ is a complex vector space and 
\begin{equation*}
OL_{X,0}:H^1(X,\Omega^{1,an}_{\log,0})=\Pic(X)\to
\mathbb H^2(X,F^1\Omega_X^{\bullet})\to\mathbb H^2(X,\Omega_X^{\bullet})
\end{equation*}
is the cycle class map.
\end{rem}

We will also consider :

\begin{defi}\label{Omegane}
We consider the (non full) subcategory $\AnSm(\mathbb C)^{pr}\subset\AnSp(\mathbb C)$
whose set of objects consists of $X\times\mathbb A^{1,an}$, $X\in\AnSm(\mathbb C)$ and 
\begin{equation*}
\Hom_{\AnSm(\mathbb C)^{pr}}(X\times\mathbb A^{1,an},Y\times\mathbb A^{1,an}):=
\Hom_{\AnSp(\mathbb C)}(X,Y)\otimes I_{\mathbb A^{1,an}}. 
\end{equation*}
We have then the morphism of sites
\begin{itemize}
\item $\pi:\AnSm(\mathbb C)\to\AnSm(\mathbb C)^{pr}$, $\pi(X\times\mathbb A^{1,an}):=X$, $\pi(f\otimes I):=f$,
\item $\rho:\AnSm(\mathbb C)^{pr}\to\AnSm(\mathbb C)$, $\rho(X):=X\times\mathbb A^{1,an}$, $\rho(f):=f\otimes I$,
\end{itemize}
in particular $\rho\circ\pi=I$.
We have then, for each $p\in\mathbb Z$, a canonical splitting in $\PSh(\AnSm(\mathbb C)^{pr})$
\begin{equation*}
\Omega^{p,an,\partial=0}=\rho_*\Omega^{p,an,\partial=0}\oplus\Omega^{p,e}, 
\end{equation*}
with
\vskip -1cm
\begin{eqnarray*}
\Omega^{p,e}(X\times\mathbb A^1):=\left\{\sum_{n=1}^{\infty}a_nz^nw(x)+w'\wedge dz,
a_n\in\mathbb C,w\in\Omega^p(X),w'\in\Omega^{p-1}(X\times\mathbb A^1)\right\} \\
\subset\Omega^{p,an,\partial=0}(X\times\mathbb A^1)
\end{eqnarray*}
where $z$ is the coordinate of $\mathbb A^1$, $a_nr^n\to 0$ when $n\to\infty$ for all $r\in\mathbb R^+$, and $x\in X$.
We have the factorization in $\PSh(\AnSm(\mathbb C))$
\begin{eqnarray*}
\ad(\pi):=\ad(\pi^*,\pi_*)(-):\pi^*\Omega^{p,e}\xrightarrow{\subset}
\pi^*\Omega^{p,an,\partial=0}:=\pi^*\pi_*\Omega^{p,an,\partial=0} \\
\xrightarrow{(0,\ad(\pi))}\ad(\pi)(\pi^*\Omega^{p,e})\xrightarrow{i_{\ad(\pi)}}\Omega^{p,an,\partial=0},
\end{eqnarray*}
and we denote again by abuse 
\begin{itemize}
\item $\pi^*\Omega^{\bullet,e}:=\ad(\pi)(\pi^*\Omega^{\bullet,e})\subset
\Omega^{\bullet,an,\partial=0}\subset\Omega^{\bullet,an}$,
\item $\pi^*\Omega^{\bullet,e}:=\ad(\pi)(\pi^*\Omega^{\bullet,e})\cap\Omega^{\bullet,an}_{\log}
\subset\Omega^{p,an,\partial=0}\subset\Omega^{\bullet,an}$.
\end{itemize}
We will consider, referring to definition \ref{omegazero},
for each $p\in\mathbb Z$, the exact sequence in $\PSh(\AnSm(\mathbb C))$
\begin{equation*}
0\to\Omega^{p,an}_{\log,0}\cap q_L^{-1}(\pi^*\Omega^{p,e})\xrightarrow{e}\Omega^{p,an}_{\log,0}\xrightarrow{q}
\Omega^{p,an,ne}_{\log,0}:=\Omega^{p,an}_{\log,0}/(\Omega^{p,an}_{\log,0}\cap q_L^{-1}(\pi^*\Omega^{p,e}))\to 0.
\end{equation*}
We have then the commutative diagram in $C(\AnSm(\mathbb C))$
\begin{equation*}
\xymatrix{\Omega_{\log,0}^{\bullet,an}\cap q_L^{-1}(\pi^*\Omega^{\bullet,e})
\ar[d]_e\ar[rr]_{OL_0^{an}} & \, & \pi^*\Omega^{\bullet,e}\ar[d]^e \\
\Omega^{\bullet,an}_{\log,0}\ar[d]_q\ar[rr]_{OL_0^{an}} & \, & \Omega^{\bullet,an}\ar[d]^q \\
\Omega^{\bullet,an,ne}_{\log,0}:=
\Omega^{\bullet,an}_{\log,0}/(\Omega_{\log,0}^{\bullet,an}\cap q_L^{-1}(\pi^*\Omega^{\bullet,e}))
\ar[rr]^{\bar OL_0^{an}} & \, &  \Omega^{\bullet,an}/(\pi^*\Omega^{\bullet,e})},
\end{equation*}
whose columns are exact with $e$ injective and $q$ surjective, and of course all differentials
of $\Omega_{\log,0}^{\bullet,an}$ and of $\pi^*\Omega^{\bullet,e}$ are trivial.
\end{defi}

\begin{lem}\label{splitting}
We have a splitting in $C(\AnSm(\mathbb C))$
\begin{equation*}
a_{usu}\Omega_{\log,0}^{\bullet,an}=
a_{usu}(\Omega_{\log,0}^{\bullet,an}\cap q_L^{-1}(\pi^*\Omega^{\bullet,e}))\oplus 
a_{usu}\Omega^{\bullet,an,ne}_{\log,0}
\end{equation*}
\end{lem}

\begin{proof}
We have
\begin{equation*}
\pi^*\Omega^{\bullet}(X):=
\varinjlim_{m:X\to Y\times\mathbb A^1}\Omega(m)(\Omega^{\bullet,an}(Y\times\mathbb A^1))
\subset\Omega^{\bullet,an}(X),
\end{equation*}
where we denote again by abuse 
\begin{equation*}
\pi^*\Omega^{\bullet}:=\ad(\pi)(\pi^*\Omega^{\bullet})\subset\Omega^{\bullet,an}.
\end{equation*}
Now, 
\begin{itemize}
\item if $q\circ m:X\xrightarrow{m}Y\times\mathbb A^1\xrightarrow{q}\mathbb A^1$ is non constant,
we can factor 
\begin{equation*}
m=(n\times I)\circ m':X\xrightarrow{m'}Y'\times\mathbb A^1\xrightarrow{n\times I}Y\times\mathbb A^1
\end{equation*}
with $\dim Y'=\dim X-1$ and $m'(X)\subset Y'\times\mathbb A^1$ an open subset, and 
\begin{equation*}
\Omega(m'):\Omega^{\bullet,an}(Y'\times\mathbb A^1)\to\Omega^{\bullet,an}(X)
\end{equation*}
so that we still have a splitting in the image of $\Omega(m')$
\item if $q\circ m:X\xrightarrow{m}Y\times\mathbb A^1\xrightarrow{q}\mathbb A^1$ is constant with value $c$,
\begin{equation*}
\Omega(m)(\pi^*\Omega^{\bullet,e})(Y\times\mathbb A^1)=0 
\end{equation*}
since $c^*dz=0$, where $c:Y\times\left\{c\right\}\hookrightarrow Y\times\mathbb A^1$ is the closed embedding.
\end{itemize}
We then define
\begin{eqnarray*}
(\pi^*\Omega^{\bullet})^{nc}\subset\pi^*\Omega^{\bullet}, \;
(\pi^*\Omega^{\bullet})^{nc}(X):=\varinjlim_{m:X\to Y\times\mathbb A^1, q\circ m \, \mbox{nc} \,}
\Omega(m)(\Omega^{\bullet,an}(Y\times\mathbb A^1))\subset\pi^*\Omega^{\bullet}(X),
\end{eqnarray*}
where $nc$ means non constant.
Now consider a splitting of complex vector spaces
\begin{equation*}
\Omega^{1,\partial=0}(D^1)=\Omega^{1,\partial=0}(\mathbb A^1)\oplus\Omega^{1,na}(D^1).
\end{equation*}
We consider similarly the (non full) subcategory $\AnSm(\mathbb C)^{pr'}\subset\AnSp(\mathbb C)$
whose set of objects consists of $X\times D^{1}$, $X\in\AnSm(\mathbb C)$ and 
\begin{equation*}
\Hom_{\AnSm(\mathbb C)^{pr}}(X\times D^{1},Y\times D^{1}):=\Hom_{\AnSp(\mathbb C)}(X,Y)\otimes I_{D^{1}}. 
\end{equation*}
We have then similarly the morphism of sites
\begin{itemize}
\item $\pi':\AnSm(\mathbb C)\to\AnSm(\mathbb C)^{pr'}$, $\pi'(X\times D^{1}):=X$, $\pi'(f\otimes I):=f$,
\item $\rho':\AnSm(\mathbb C)^{pr'}\to\AnSm(\mathbb C)$, $\rho'(X):=X\times D^{1}$, $\rho'(f):=f\otimes I$,
\end{itemize}
in particular $\rho'\circ\pi'=I$. We define similarly 
\begin{eqnarray*}
(\pi^{'*}\Omega^{\bullet})^{nc}\subset\pi^{'*}\Omega^{\bullet}, \;
(\pi^{'*}\Omega^{\bullet})^{nc}(X):=\varinjlim_{m:X\to Y\times D^1, q\circ m \, \mbox{nc} \,}
\Omega(m)(\Omega^{\bullet,an}(Y\times D^1))\subset\pi^{'*}\Omega^{\bullet}(X).
\end{eqnarray*}
We have then
\begin{eqnarray*}
\Omega^{\bullet,an}=a_{usu}(\pi^{'*}\Omega^{\bullet})^{nc}=
a_{usu}(\pi^*\Omega^{\bullet})^{nc}\oplus a_{usu}\pi^{'*}\Omega^{\bullet,nc,na}, \\
\pi^{'*}\Omega^{\bullet,na}(X):=\varinjlim_{m:X\to Y\times D^1, q\circ m \, \mbox{nc} \,}
\Omega(m)(\Omega^{\bullet-1,an}(Y)\times\Omega^{1,na}(D^1))\subset\Omega^{\bullet,an}(X)
\end{eqnarray*}
and thus
\begin{eqnarray*}
a_{usu}\Omega_{\log,0}^{\bullet,an}=a_{usu}q_L^{-1}(\Omega_{\log}^{\bullet,an}\cap(\pi^{'*}\Omega^{\bullet})^{nc})
=a_{usu}(\Omega_{\log,0}^{\bullet,an}\cap q_L^{-1}(\pi^*\Omega^{\bullet})^{nc})\oplus 
a_{usu}(\Omega_{\log,0}^{\bullet,an}\cap q_L^{-1}(\pi^{'*}\Omega^{\bullet,nc,na})),
\end{eqnarray*}
where the first equalities follows from the fact for an open ball $D^{d_X}\subset X$, 
we have the morphism $m_I:D^{d_X}\to D^{d_X-1}\times D^1$ which is the canonical isomorphism.
Hence,
\begin{equation*}
a_{usu}\Omega_{\log,0}^{\bullet,an}=a_{usu}(\Omega_{\log,0}^{\bullet,an}\cap q_L^{-1}(\pi^*\Omega^{\bullet,e}))\oplus
\tilde\Omega^{\bullet,an,ne}_{\log,0}
\end{equation*}
with
\begin{equation*}
\tilde\Omega^{\bullet,an,ne}_{\log,0}(X):=
a_{usu}(\Omega_{\log,0}^{\bullet,an}\cap q_L^{-1}(\pi^*\rho_*\Omega^{\bullet})^{nc})\oplus
a_{usu}(\Omega_{\log,0}^{\bullet,an}\cap q_L^{-1}(\pi^{'*}\Omega^{\bullet,nc,na})).
\end{equation*}
\end{proof}

We have the following :

\begin{lem}\label{a1trcan}
\begin{itemize}
\item[(i0)]The sheaves $O_{an}^*\in\PSh(\AnSm(\mathbb C))$ and $\Omega^{1,an}_{\log}\in\PSh(\AnSm(\mathbb C))$ 
admit transfers compatible with transfers on $\Omega^{1,an}\in\PSh(\AnSm(\mathbb C))$.
\item[(i)] For each $l\in\mathbb Z$, the sheaf $\Omega^{l,an}_{\log}\in\PSh(\AnSm(\mathbb C))$ admits transfers
compatible with transfers on $\Omega^{l,an}\in\PSh(\AnSm(\mathbb C))$,
that is $\Omega^{\bullet,an}\in C(\Cor\AnSm(\mathbb C))$  and the inclusion 
$OL:\Omega^{l,an}_{\log}[-l]\hookrightarrow\Omega^{\bullet,an}$ in $C(\AnSm(\mathbb C))$ 
is compatible with transfers.
\item[(ii)]For each $l\in\mathbb Z$, the presheaf 
$\Omega^{l,an,ne}_{\log,0}\in\PSh(\AnSm(\mathbb C))$ given in definition \ref{Omegane} is $\mathbb A^1$ invariant.
\item[(ii)']For each $l\in\mathbb Z$, the sheaf 
$a_{usu}\Omega^{l,an,ne}_{\log,0}\in\PSh(\AnSm(\mathbb C))$ given in definition \ref{Omegane} is $\mathbb A^1$ invariant,
where $a_{usu}:\PSh(\AnSp(\mathbb C))\to\Shv(\AnSp(\mathbb C))$ is the sheaftification functor.
\end{itemize}
\end{lem}

\begin{proof}
\noindent(i0):Let $W\subset X'\times X$ with $X,X'\in\AnSm(\mathbb C)$, $X'$ connected, and 
$p_{X'}:W\subset X'\times X\to X'$ finite surjective over $X'$. 
Then $M(p_X):M(X')\hookrightarrow M(W)$ is a finite, 
where for $Y\in\AnSp(\mathbb C)$ irreducible, $M(Y):=Frac(O(Y))$ denote the field of meromorphic function.
Since $O(X')\subset M(X')$ is integrally closed, the trace map $Tr_{W/X'}:M(W)\to M(X')$ sends $O(W)$ to $O(X')$, 
and the norm map $N_{W/X'}:M(W)^*\to M(X')^*$ sends $O(W)^*$ to $O(X')^*$.

The sheaf $O_{an}^*\in\PSh(\AnSm(\mathbb C))$ admits transfers : 
for $W\subset X'\times X$ with $X,X'\in\AnSm(\mathbb C)$
and $W$ finite surjective over $X'$ and $f\in O(X)^*$, $W^*f:=N_{W/X'}(p_X^*f)$ where $p_X:W\hookrightarrow X'\times X\to X$
is the projection and $N_{W/X'}:M(W)^*\to M(X')^*$ is the norm map.
This gives transfers on $\Omega^{1,an}_{\log}\in\PSh(\AnSm(\mathbb C))$ compatible with transfers on 
$\Omega^{1,an}\in\PSh(\AnSm(\mathbb C))$ :
for $W\subset X'\times X$ with $X,X'\in\AnSm(\mathbb C)$ and $W$ finite surjective over $X'$ and $f\in O(X)^*$, 
\begin{equation*}
W^*df/f:=dW^*f/W^*f=Tr_{W/X'}(p_X^*(df/f)), 
\end{equation*}
where we recall $p_X:W\hookrightarrow X'\times X\to X$ is the projection and $Tr_{W/X'}:O(W)\to O(X')$ is the trace map. 
Note that $d(fg)/fg=df/f+dg/g$. 

\noindent(i): By (i0), we get transfers on 
\begin{equation*}
\otimes^l_{\mathbb Q}\Omega^{1,an}_{\log}, \, \otimes_{O}^l\Omega^{1,an}\in\PSh(\AnSm(\mathbb C))
\end{equation*}
since $\otimes^l_{\mathbb Q}\Omega^{1,an}_{\log}=H^0(\otimes^{L,l}_{\mathbb Q}\Omega^{1,an}_{\log})$ and 
$\otimes_O^l\Omega^{1,an}=H^0(\otimes_O^{L,l}\Omega^{1,an})$.
This induces transfers on  
\begin{equation*}
\wedge^l_{\mathbb Q}\Omega^{1,an}_{\log}:=
\coker(\oplus_{I_2\subset[1,\ldots,l]}\otimes^{l-1}_{\mathbb Q}\Omega^{1,an}_{\log}
\xrightarrow{\oplus_{I_2\subset[1,\ldots,l]}\Delta_{I_2}:=(w\otimes w'\mapsto w\otimes w\otimes w')}
\otimes^l_{\mathbb Q}\Omega^{1,an}_{\log})
\in\PSh(\AnSm(\mathbb C)).
\end{equation*}
and
\begin{equation*}
\wedge^l_O\Omega^{1,an}:=
\coker(\oplus_{I_2\subset[1,\ldots,l]}\otimes^{l-1}_{O_k}\Omega^{1,an}
\xrightarrow{\oplus_{I_2\subset[1,\ldots,l]}\Delta_{I_2}:=(w\otimes w'\mapsto w\otimes w\otimes w')})
\otimes^l_O\Omega^{1,an}
\in\PSh(\AnSm(\mathbb C)).
\end{equation*}

\noindent(ii): Let $X\in\AnSm(\mathbb C)$ and, denoting $p:X\times\mathbb A^1\to X$ the projection,
\begin{equation*}
\Omega(p):\Omega_{\log,0}^{p,an,ne}(X)\to\Omega_{\log,0}^{p,an,ne}(X\times\mathbb A^1).
\end{equation*}
Since $\Omega(p)$ is split injective, it is enough to show that $\Omega(p)$ is surjective. We have
\begin{equation*}
\Omega_{\log,0}^{p,an,ne}(X\times\mathbb A^1):=
\Omega_{\log,0}^{p,an}(X\times\mathbb A^1)/q_L^{-1}\Omega^{p,e}(X\times\mathbb A^1)
\end{equation*}
The result then follows from the fact that for $X\in\AnSp(\mathbb C)$, 
$\Omega^p(X\times\mathbb A^1)=\Omega^p(X)\oplus(\Omega^{p-1}(X)\otimes\Omega^1(\mathbb A^{1,an}))$ and that
$\mathbb A^{1,an}$ is simply connected. 
Indeed, for $w=w'+w''\wedge dz\in\Omega^p(X\times\mathbb A^1)^{\partial=0}$, 
\begin{equation*}
w(x,z)=\partial(\int_{t=0}^{t=z}w''(x,t)dt)+w(x,0)=
\partial(\int_{t=0}^{t=z}w''(x,t)dt)_e+\partial(\int_{t=0}^{t=z}w''(x,t)dt)(x,0)+w'(x,0)
\end{equation*}
and $\partial(\int_{t=0}^{t=z}w''(x,t)dt)_e\in\Omega^{p,e}(X\times\mathbb A^1)$

\noindent(ii): Follows from (ii) since open ball form a basis of the complex topology on complex analytic manifolds.
\end{proof}

We will use the following lemma

\begin{lem}\label{Oan}
\begin{itemize}
\item[(i0)]Let $D\in\AnSm(\mathbb C)$ be an open ball.  
Then, for each $p,q\in\mathbb Z$, $q\neq 0$, $H_{usu}^q(D,\Omega^p_{D,\log,0})=0$ and $H_{usu}^q(D,\Omega^p_{D\times\mathbb A^n,\log,0})=0$.
\item[(i1)]Let $U\subset\mathbb A^N$ be a smooth affine complex algebraic variety.
We have for $p,q\in\mathbb Z$, $q\neq 0$, $p\geq 1$, $H^q(U^{an},O_{U^{an}}^{\otimes p})=0$.
\item[(i2)]Let $U\subset\mathbb A^N$ be a smooth affine complex algebraic variety.
We have for $p,q\in\mathbb Z$, $q\neq 0$, $p\geq 1$, $H^q(U^{an},\wedge^{p-1}O^*_{U^{an}}\otimes O_{U^{an}})=0$.
\end{itemize}
\end{lem}

\begin{proof}
\noindent(i0):Let $B=D(0,1)^N\in\AnSm(\mathbb C)$ be an open ball. 
Let $\alpha\in H_{usu}^q(B,\Omega^p_{B,\log,0})$, $q\in\mathbb Z$, $q\neq 0$.
Consider a sequence of sub-balls 
\begin{equation*}
l_n:B_n:=D(0,r_n)^N\hookrightarrow B:=D(0,1)^N, n\in\mathbb N, \; \mbox{with} \; 
r_n\to 0 \; \mbox{when} \; n\to\infty.
\end{equation*}
Since $(B_n:=D(0,r_n))_{n\in\mathbb N}$ form a basis of neighborhood of $0\in B$ and 
since $\Omega^p_{B,\log,0}$ is a single presheaf, there exists $n:=n_{\alpha}\in\mathbb N$ (depending on $\alpha$) such that 
\begin{equation*}
l_n^*\alpha=0\in H_{usu}^q(B_n,\Omega^p_{B,\log,0}). 
\end{equation*}
But it follows immediately from the definition of the the presheaf $\Omega^p_{B,\log,0}$ considering the isomorphism 
\begin{equation*}
\phi_n:B_n\xrightarrow{\sim}B, \; \phi_n(z):=r_n^{-1}z
\end{equation*}
so that $(\phi_n)^{-1}\circ l_n=r_nI$, that 
\begin{equation*}
l_n^*\alpha=l_n^*\phi_n^{-1*}\phi_n^*\alpha=r_n\phi_n^*\alpha\in H_{usu}^q(B_n,\Omega^p_{B,\log,0}).
\end{equation*}
Hence $\phi_n^*\alpha=0$. Since $\phi_n$ is an isomorphism, we get $\alpha=0$.

\noindent(i1):Denote by $\Delta_U:U\hookrightarrow\Pi_{k=1}^pU$ the diagonal,
where $\Pi_{k=1}^pU=U\times_{\Spec(\mathbb Z)}\cdots\times U$ is the product of schemes (over $\Spec Z$).
We have then a canonical isomorphism of $O_U$ module $O_U^{\otimes p}\xrightarrow{\sim}\Delta_U^*O_{\Pi_{k=1}^pU}$ on $U$, 
where the action of $O_U$ on $O_U^{\otimes p}$ is given by the diagonal $h\cdot(h_1,\cdots h_p):=(hh_1,\cdots,hh_p)$.
Hence $O_U^{\otimes p}$ is a quasi-coherent $O_U$ module on $U$.
Thus, denoting by $an_U:U^{an}\to U$ the analytification map, we have 
\begin{equation*}
H^q(U^{an},O_{U^{an}}\otimes\an_U^*O_U^{\otimes p-1})=0.
\end{equation*}
On the other hand, take an (may be infinite) cover $\mathbb A^{N,an}=\cup_{i\in\Lambda}D_i$ by open balls $D_i\subset\mathbb A^N$ 
such that $\phi_{D_i}:U\cap D_i\simeq V(z_1,z_{N-d_U})$. We have then using (i0) 
\begin{eqnarray*}
H^q(U^{an},O_{U^{an}}\otimes\an_U^*O_U^{\otimes p-1})\otimes_{O(\mathbb A^N)}\otimes O(\mathbb A^{N,an})^{\otimes p-1}= \\
H^q(\Gamma(D_{\bullet},O_{U^{an}}\otimes\an_U^*O_U^{\otimes p-1})\otimes_{O(\mathbb A^N)}\otimes O(\mathbb A^{N,an})^{\otimes p-1})=
H^q\Gamma(D_{\bullet},O_{U^{an}}^{\otimes p})=H^q(U^{an},O_{U^{an}}^{\otimes p}).
\end{eqnarray*}
Hence $H^q(U^{an},O_{U^{an}}^{\otimes p})=0$.

\noindent(i2):Follows from (i1) by induction on $p$ using the exponential exact sequences on $U^{an}$
\begin{equation*}
0\to\oplus_{k=1}^{p-1}{O^*_{U^{an}}}^{\otimes p-2}\otimes O_{U^{an}}\to {O^*_{U^{an}}}^{\otimes p-2}\otimes O_{U^{an}}^{\otimes 2}
\xrightarrow{e\otimes I}{O^*_{U^{an}}}^{\otimes p-1}\otimes O_{U^{an}}\to 0
\end{equation*}
that $H^q(U^{an},\wedge^{p-1}O^*_{U^{an}}\otimes O_{U^{an}})=0$.
\end{proof}

We have the following result :

\begin{prop}\label{UXusu}
Consider, for each $p\in\mathbb Z$, 
the presheaf $\Omega^{p,an,ne}_{\log}\in\PSh(\AnSm(\mathbb C))$ given in definition \ref{Omegane}.
\begin{itemize}
\item[(i)] For each $p,q\in\mathbb Z$,
\begin{equation*}
H^qE_{usu}(\Omega_{\log,0}^{p,an,ne})\in\PSh(\AnSm(\mathbb C)) 
\end{equation*}
is $\mathbb A^1$ invariant. Hence, by proposition \ref{smXZ}, we get for $X\in\SmVar(\mathbb C)$ and $Z\subset X$ a smooth closed subset of pure codimension $c$
a canonical map for each $p,q\in\mathbb Z$
\begin{equation*}
\Omega(Res_{Z,X}):H_{usu}^q(X^{an}\backslash Z,\Omega_{\log,0}^{p,an,ne})\to H^{q+1}_{usu,Z}(X^{an},\Omega^{p,an,ne}_{\log,0})\to
H_{usu}^{q+1-c}(Z^{an},\Omega^{p-c,an,ne}_{\log,0}),
\end{equation*}
where $Res_{Z,X}:M(Z)(c)[-2c]\to M_Z(X)\to M(X\backslash Z)[1]$ is the map given in definition \ref{ResMot}.
\item[(ii)]Let $X\in\PSmVar(\mathbb C)$ connected of dimension $d_U$.
Let $j:U\hookrightarrow X$ be an open embedding such that $\Pic(U)=0$ (where $\Pic(U)$ is the Picard group of the algebraic variety $U$) 
and $X\backslash U:=D=D'\cup E$ where $D'\subset X$ is an ample divisor and $E\subset X$ is a smooth ample divisor (use Bertini for generic smoothness), 
we may indeed add components if necessary since the map on picard groups induced by an open embedding of scheme is surjective by taking zariski closure.
Then, for each $p,q\in\mathbb Z$, $p\geq q$, $j^*H_{usu}^q(X^{an},\Omega^{p,an,ne}_{\log,0})=0$.
\end{itemize}
\end{prop}

\begin{proof}
\noindent(i):Consider the splitting
\begin{equation*}
a_{usu}\Omega_{\log,0}^{\bullet,an}=
a_{usu}(\Omega_{\log,0}^{\bullet,an}\cap q_L^{-1}(\pi^*\Omega^{\bullet,e}))\oplus a_{usu}\Omega^{\bullet,an,ne}_{\log,0}
\end{equation*}
given in lemma \ref{splitting}. It gives in particular a canonical isomorphism
\begin{equation*}
H_{usu}^q(D,\Omega^p_{\log,0})=H_{usu}^q(D,\Omega^p_{\log,0}\cap\pi^*\Omega^{p,e})
\oplus H_{usu}^q(D,\Omega^{p,an,ne}_{\log,0}).
\end{equation*}
By lemma \ref{Oan}(i0) we have
\begin{equation*}
H_{usu}^q(D,\Omega^p_{\log,0})=0=
H_{usu}^q(D,\Omega^p_{\log,0}\cap\pi^*\Omega^{p,e})\oplus H_{usu}^q(D,\Omega^{p,an,ne}_{\log,0}).
\end{equation*}
Hence $H_{usu}^q(D,\Omega^{p,an,ne}_{\log,0})=0$. 
Similarly, by lemma \ref{Oan}(i0) and the splitting we also have $H_{usu}^q(\mathbb A_{\mathbb C}^n\times D,\Omega^{p,an,ne}_{\log,0})=0$.
This proves (i). One can also prove (i) directly using the same arguments as in the proof of lemma \ref{Oan}(i0).
Let $X\in\AnSm(\mathbb C)$ connected. Take an open cover $X=\cup_{i\in I}D_i$ by open balls $D_i\simeq D(0,1)^{d_X}$, 
and let $D_{\bullet}\in\Fun(\Delta,\AnSm(\mathbb C))$ be the associated simplicial complex manifold,
with for $J\subset I$, $j_{IJ}:D_I:=\cap_{i\in I} D_i\hookrightarrow D_J:=\cap_{i\in J}D_i$ the open embedding. 
Then we get the isomorphism
\begin{eqnarray*}
p_X^*:H^qE_{usu}(\Omega_{\log,0}^{p,an,ne})(X\times\mathbb A^1):=
H^q_{usu}(X\times\mathbb A^1,\Omega_{\log,0}^{p,an,ne})\xrightarrow{\sim} \\
\mathbb H^q(D_{\bullet}\times\mathbb A^1,\Omega_{\log,0}^{p,an,ne})\xrightarrow{\sim} 
H^q\Gamma(D_{\bullet}\times\mathbb A^1,a_{usu}\Omega_{\log,0}^{p,an,ne}) \\
\xrightarrow{\sim}H^q\Gamma(D_{\bullet},a_{usu}\Omega_{\log,0}^{p,an,ne})
\xrightarrow{\sim}\mathbb H^q(D_{\bullet},\Omega_{\log,0}^{p,an,ne})
\xrightarrow{\sim}H^q_{usu}(X,\Omega_{\log,0}^{p,an,ne})=:H^qE_{usu}(\Omega_{\log}^{p,an,ne})(X)
\end{eqnarray*}
where $a_{usu}$ is the sheafification functor, $p_X:X\times\mathbb A^1\to X$ is the projection, 
and the middle map is an isomorphism by lemma \ref{a1trcan}(ii)'.

\noindent(ii):For $X\in\PSmVar(\mathbb C)$ and $j:U\hookrightarrow X$ satisfying the hypothesis of (ii)
Denote by $j':X\backslash D'\hookrightarrow X$ the open embedding.
Denote by $F:=E\cap D'$. We denote again by $j':X\backslash D'\hookrightarrow X\backslash F$ and $j:E\backslash F\hookrightarrow E$ the open embeddings 
given by the complement of $D'\backslash F$ and $F=D'\cap E$ respectively and by 
$l:X\backslash E\hookrightarrow X\backslash F$ the open embedding given by the complement of $E$.
We will prove by induction on $q\geq 1$ that for all $X\in\PSmVar(\mathbb C)$ and $j:U\hookrightarrow X$ satisfying the hypothesis of (ii), 
we have for all $p\geq q$
\begin{equation*}
j^*H^q(X^{an}\backslash F,\Omega_{\log,0}^{p,an,ne})=0.
\end{equation*}
For $X'\in\SmVar(\mathbb C)$ and $Z\subset X'$ a closed subset, we denote, for each $p,q\in\mathbb Z$, by
\begin{eqnarray*}
c_{e,(X',Z)}:H^q_Z(X^{'an},\wedge^pO_{X^{'an}}^*)\to H^{q+1}_Z(X^{'an},\wedge^{p-1}O_{X^{'an}}^*), \\
c_{e,(X',Z)}:={c_{e,(X',Z)}}_{|H^q_Z(X^{'an},\Omega_{\log,0}^{p,an,ne})}:
H^q_Z(X^{'an},\Omega_{\log,0}^{p,an,ne})\to H^{q+1}_Z(X^{'an},\Omega_{\log,0}^{p-1,an,ne})
\end{eqnarray*}
the maps induced by the exponential exact sequence
\begin{equation*}
0\to {O_{X^{'an}}^*}^{\otimes p-1}\to {O_{X^{'an}}^*}^{\otimes p-1}\otimes O_{X^{'an}}\xrightarrow{I\otimes e}{O_{X^{'an}}^*}^{\otimes p}\to 0
\end{equation*}
and the lemma \ref{splitting}.
Let $X\in\PSmVar(\mathbb C)$ and $j:U\hookrightarrow X$ satisfying the hypothesis of (ii)
We have, for each $p,q\in\mathbb Z$, using (i), the following commutative diagrams
\begin{eqnarray*}
\xymatrix{\Omega(Res_{E,X}):H^q(X^{an}\backslash E,\Omega_{\log,0}^{p,an,ne})\ar[r]^{\Omega(-)} & H^{q+1}_E(X^{an},\Omega^{p,an,ne}_{\log,0})\ar[r]^{\Omega(-)} &
H^q(E^{an},\Omega^{p-1,an,ne}_{\log,0}) \\
\Omega(Res_{E,X}):H^{q-1}(X^{an}\backslash E,\Omega_{\log,0}^{p+1,an,ne})\ar[r]^{\Omega(-)}\ar[u]^{c_{e,X\backslash E}} & 
H^q_E(X^{an},\Omega^{p+1,an,ne}_{\log,0})\ar[r]^{\Omega(-)}\ar[u]^{c_{e,(X,E)}} & H^{q-1}(E^{an},\Omega^{p,an,ne}_{\log,0})\ar[u]^{c_{e,E}}}, \\
\end{eqnarray*}
and
\begin{eqnarray*}
\xymatrix{H^{q-1}(X^{an}\backslash E,\Omega_{\log,0}^{p,an,ne})\ar[r]^{\Omega(Res_{E,X})}\ar[d]^{j^{'*}:=\Omega(j')} &
H^{q-1}(E^{an},\Omega^{p-1,an,ne}_{\log,0})\ar[d]^{j^{'*}:=\Omega(j')} \\
H^{q-1}(U^{an}\backslash E,\Omega_{\log,0}^{p+1,an,ne})\ar[r]^{\Omega(Res_{E\backslash F,X\backslash F})} & H^{q-1}(E^{an}\backslash F,\Omega^{p,an,ne}_{\log,0})}, \;
\xymatrix{H^q(E^{an},\Omega_{\log,0}^{p-1,an,ne})\ar[r]^{j^{'*}:=\Omega(j')} & H^q(E^{an}\backslash F,\Omega_{\log,0}^{p-1,an,ne}) \\
H^{q-1}(E^{an},\Omega_{\log,0}^{p,an,ne})\ar[r]^{j^{'*}:=\Omega(j')}\ar[u]^{c_{e,E}} & H^{q-1}(E^{an}\backslash F,\Omega_{\log,0}^{p,an,ne})\ar[u]^{c_{e,E\backslash F}}},
\end{eqnarray*}
Let $\alpha\in H^q(X^{an}\backslash F,\Omega^{p,an,ne}_{\log,0})$. Then $\Omega(Res_{E\backslash F,X\backslash F})l^*\alpha=0$. 
Consider the 
\begin{equation*}
\alpha'\in H^{q-1}(X^{an}\backslash E,\Omega_{\log,0}^{p+1,an,ne}) \mbox{\; such that \;} l^*\alpha=c_{e,X\backslash E}\alpha', 
\end{equation*}
$c_{e,X\backslash E}$ being an isomorphism by lemma \ref{Oan}(i2) since $X\backslash E$ is affine. Then 
\begin{equation*}
j^{'*}c_{e,E}\Omega(Res_{E,X})\alpha'=j^{'*}\Omega(Res_{E,X})c_{e,X\backslash E}\alpha'=j^{'*}\Omega(Res_{E,X})l^*\alpha=\Omega(Res_{E\backslash F,X\backslash F})l^*\alpha=0.
\end{equation*}
We get
\begin{equation*}
c_{e,E\backslash F}j^{'*}\Omega(Res_{E,X})\alpha'=j^{'*}c_{e,E}\Omega(Res_{E,X})\alpha'=0.
\end{equation*}
Since $c_{e,E\backslash F}$ is an isomorphism by lemma \ref{Oan}(i2) since $E\backslash F$ is affine ($D\subset X$ is ample), 
we get $j^{'*}\Omega(Res_{E,X})\alpha'=0$. Hence 
\begin{equation*}
\Omega(Res_{E\backslash F,X\backslash F})j^{'*}\alpha'=j^{'*}\Omega(Res_{E,X})\alpha'=0.
\end{equation*}
Hence $j^{'*}\alpha'=l^*\beta$, with $\beta\in H^{q-1}(U^{an},\Omega_{\log,0}^{p+1,an,ne})$.
By the Mayer-Vietoris long exact sequence for $X\backslash F=(X\backslash E)\cup (X\backslash D')$
\begin{eqnarray*}
\cdots\to H^{q-1}(X^{an}\backslash F,\Omega_{\log,0}^{p+1,an,ne})\xrightarrow{(l^*,j^{'*})}
H^{q-1}(X^{an}\backslash E,\Omega_{\log,0}^{p+1,an,ne})\oplus H^{q-1}(X^{an}\backslash D',\Omega_{\log,0}^{p+1,an,ne}) \\
\xrightarrow{l^*-j^{'*}} H^{q-1}(X^{an}\backslash (E\cup D'),\Omega_{\log,0}^{p+1,an,ne})\to\cdots,
\end{eqnarray*}
there exists $\gamma\in H^{q-1}(X^{an}\backslash F,\Omega_{\log,0}^{p+1,an,ne})$ such that $\alpha'=l^*\gamma$. Hence, 
\begin{equation}\label{eqE}
\alpha=c_{e,X\backslash F}\gamma+\eta, \mbox{\; with \;} 
\eta\in H_E^q(X^{an}\backslash F,\Omega_{\log,0}^{p,an,ne})\simeq H^{q-1}(E^{an}\backslash F,\Omega_{\log,0}^{p-1,an,ne}).
\end{equation}
Hence
\begin{itemize}
\item If $q=1$. Let $X\in\PSmVar(\mathbb C)$ and $j:U\hookrightarrow X$ satisfying the hypothesis of (ii). 
Let $\alpha\in H^1(X^{an}\backslash F,\Omega_{\log,0}^{p,an,ne})$, $p\geq 1$.
We have (\ref{eqE})
\begin{equation*}
\alpha=c_{e,X\backslash F}\gamma+\eta, \mbox{\; with \;} \gamma\in H^0(X^{an}\backslash F,\Omega_{\log,0}^{p+1,an,ne}), \;
\eta\in H_E^1(X^{an}\backslash F,\Omega_{\log,0}^{p,an,ne})\simeq H^0(E^{an}\backslash F,\Omega_{\log,0}^{p-1,an,ne}).
\end{equation*}
Since $\Pic(X^{an}\backslash F)=\Pic(X\backslash F)$ by GAGA and since $\Pic(U)=0$ we get that $j^*c_{e,X\backslash F}\gamma=0$.
Thus $j^*\alpha=0$.
\item Let $q\geq 2$ and assume the result is true for $q-1$.
Let $X\in\PSmVar(\mathbb C)$ and $j:U\hookrightarrow X$ satisfying the hypothesis of (ii). 
Let $\alpha\in H^q(X^{an}\backslash F,\Omega_{\log,0}^{p,an,ne})$, $p\geq q$. We have (\ref{eqE})
\begin{eqnarray*}
\alpha=c_{e,X\backslash F}\gamma+\eta, \mbox{\; with \;} \gamma\in H^{q-1}(X^{an}\backslash F,\Omega_{\log,0}^{p+1,an,ne}), \\
\eta\in H_E^q(X^{an}\backslash F,\Omega_{\log,0}^{p,an,ne})\simeq H^{q-1}(E^{an}\backslash F,\Omega_{\log,0}^{p-1,an,ne}).
\end{eqnarray*}
But $j^*c_{e,X\backslash F}\gamma=0$ by induction hypothesis and $j^*\eta=0$ by induction hypothesis since $p-1\geq 1$ 
(note that $\Pic(E\backslash F)=0$ by Lefchetz hyperplane theorem since $\Pic(X\backslash D')=0$). Hence 
\begin{equation*}
j^*\alpha=j^*c_{e,X\backslash F}\gamma+j^*\eta=0.
\end{equation*}
\end{itemize}
\end{proof}

\section{Complex analytic logarithmic de Rham classes}

We recall (see section 2) the morphism of site given by the analytification functor
\begin{equation*}
\An:\AnSm(\mathbb C)\to\SmVar(\mathbb C), \; \; X\mapsto\An(X):=X^{an}, \; (f:X'\to X)\mapsto\An(f):=f^{an}.
\end{equation*}
and for $X\in\SmVar(\mathbb C)$, the following commutative diagram
\begin{equation*}
\xymatrix{\AnSm(\mathbb C)\ar[r]^{\An}\ar[d]^{o_{X^{an}}} & \SmVar(\mathbb C)\ar[d]^{o_X} \\
X^{an,et}\ar[r]^{\an_X} & X^{et}}.
\end{equation*}
We have also the morphism of sites (see definition \ref{Omegane})
\begin{itemize}
\item $\pi:\AnSm(\mathbb C)\to\AnSm(\mathbb C)^{pr}$, $\pi(X\times\mathbb A^{1,an}):=X$, $\pi(f\otimes I):=f$,
\item $\rho:\AnSm(\mathbb C)^{pr}\to\AnSm(\mathbb C)$, $\rho(X):=X\times\mathbb A^{1,an}$, $\rho(f):=f\otimes I$,
\end{itemize}
in particular $\rho\circ\pi=I$.
Consider the commutative diagram in $C(\AnSm(\mathbb C))$
\begin{equation*}
\xymatrix{\Omega_{\log,0}^{\bullet,an}\cap q_L^{-1}(\pi^*\Omega^{\bullet,e})
\ar[d]_e\ar[rr]_{OL_0^{an}} & \, & \pi^*\Omega^{\bullet,e}\ar[d]^e \\
\Omega^{\bullet,an}_{\log,0}\ar[d]_q\ar[rr]_{OL_0^{an}} & \, & \Omega^{\bullet,an}\ar[d]^q \\
\Omega^{\bullet,an,ne}_{\log,0}:=
\Omega^{\bullet,an}_{\log,0}/(\Omega_{\log,0}^{\bullet,an}\cap q_L^{-1}(\pi^*\Omega^{\bullet,e}))
\ar[rr]^{\bar OL_0^{an}} & \, &  \Omega^{\bullet,an}/(\pi^*\Omega^{\bullet,e})},
\end{equation*}
whose columns are exact with $e$ injective and $q$ surjective (see definition \ref{Omegane}).
We have the factorization in $C(\AnSm(\mathbb C))$
\begin{eqnarray*}
\ad(\pi):=\ad(\pi^*,\pi_*)(-):\pi^*\Omega^{\bullet,e}\xrightarrow{\subset}
\pi^*a_{usu}\partial\Omega^{\bullet-1,an}:=\pi^*\pi_*\Omega^{\bullet,an,\partial=0} \\
\xrightarrow{(\ad(\pi),0)}\ad(\pi)(\pi^*\Omega^{\bullet,e})\xrightarrow{i_{\ad(\pi)}}\Omega^{\bullet,an,\partial=0},
\end{eqnarray*}
and we have denoted again by abuse 
\begin{itemize}
\item $\pi^*\Omega^{\bullet,e}:=\ad(\pi)(\pi^*\Omega^{\bullet,e})\subset\Omega^{\bullet,an,\partial=0}
\subset\Omega^{\bullet,an}$.
\item $\pi^*\Omega^{\bullet,e}:=\ad(\pi)(\pi^*\Omega^{\bullet,e})\cap\Omega_{\log}^{\bullet,an}
\subset\Omega^{\bullet,an,\partial=0}\subset\Omega^{\bullet,an}$.
\end{itemize}
Let $X\in\SmVar(\mathbb C)$. Then, for $k\in\mathbb Z$, we have the commutative diagram
\begin{equation*}
\xymatrix{\mathbb H_{usu}^k(X^{an},\Omega_{\log,0}^{\bullet,an}\cap q_L^{-1}(\pi^*\Omega^{\bullet,e}))
\ar[d]_{H^k(e)}\ar[r]^{H^kOL_0^{an}} & 
\mathbb H_{usu}^k(X^{an},\ad(\pi)(\pi^*\Omega^{\bullet,e}))\ar[d]^{H^k(e)} \\
\mathbb H_{usu}^k(X^{an},\Omega^{\bullet,an}_{\log,0})\ar[d]_{H^k(q)}\ar[r]^{H^kOL_0^{an}} &  
H^k_{DR}(X)=\mathbb H_{usu}^k(X^{an},\Omega^{\bullet,an})\ar[d]^{H^k(q)} \\
\mathbb H_{usu}^k(X^{an},\Omega^{\bullet,an,ne}_{\log,0})\ar[r]^{H^k\bar OL_0^{an}} &   
\mathbb H_{usu}^k(X^{an},\Omega^{\bullet,an}/(\pi^*\Omega^{\bullet,e}))}
\end{equation*}
whose columns are exact, and of course all differentials
of $\Omega_{\log,0}^{\bullet,an}$ and of $\pi^*\Omega^{\bullet,e}$ are trivial. 
Let $Z\subset X$ a closed subset. Then, for $k\in\mathbb Z$,
we also have the commutative diagram
\begin{equation*}
\xymatrix{\mathbb H_{usu,Z}^k(X^{an},\Omega_{\log,0}^{\bullet,an}\cap q_L^{-1}(\pi^*\Omega^{\bullet,e}))
\ar[d]_{H^k(e)}\ar[r]^{H^kOL_0^{an}} & 
\mathbb H_{usu,Z}^k(X^{an},\ad(\pi)(\pi^*\Omega^{\bullet,e}))\ar[d]^{H^k(e)} \\
\mathbb H_{usu,Z}^k(X^{an},\Omega^{\bullet,an}_{\log,0})\ar[d]_{H^k(q)}\ar[r]^{H^kOL_0^{an}} &  
H^k_{DR,Z}(X)=\mathbb H_{usu}^k(X^{an},\Omega^{\bullet,an})\ar[d]^{H^k(q)} \\
\mathbb H_{usu,Z}^k(X^{an},\Omega^{\bullet,an,ne}_{\log,0})\ar[r]^{H^k\bar OL_0^{an}} &   
\mathbb H_{usu,Z}^k(X^{an},\Omega^{\bullet,an}/(\pi^*\Omega^{\bullet,e}))}
\end{equation*}
whose columns are exact.

\begin{lem}\label{cGAGAloglem}
Let $X\in\SmVar(\mathbb C)$. 
For each $k\in\mathbb Z$, $k\neq 0$, we have $H^k(e)(\mathbb H_{usu}^k(X^{an},\ad(\pi)(\pi^*\Omega^{\bullet,e})))=0$.
\end{lem}

\begin{proof}
We have the canonical splitting
\begin{equation*}
\mathbb H_{usu}^k(X^{an},\ad(\pi)(\pi^*\Omega^{\bullet,e}))=
\oplus_{p\in\mathbb Z}H_{usu}^{k-p}(X^{an},\ad(\pi)(\pi^*\Omega^{p,e}))
\end{equation*}
Consider for each $p\in\mathbb Z$, the composition in $C(\AnSm(\mathbb C))$
\begin{equation*}
\ad(\pi):=\ad(\pi^*,\pi_*)(-):\pi^*(\mathcal A^{\bullet\geq p,\bullet,e})\xrightarrow{\ad(\pi)}
\ad(\pi)(\pi^*(\mathcal A^{\bullet\geq p,\bullet,e}))\xrightarrow{i_{\ad(\pi)}}\mathcal A^{\bullet\geq p,\bullet},
\end{equation*}
where the embedding 
\begin{equation*}
\Omega^{p,e}\hookrightarrow(\Omega^{\bullet\geq p})^e\hookrightarrow\mathcal A^{\bullet\geq p,\bullet,e} 
\end{equation*}
in $C(\AnSm(\mathbb C)^{pr})$ is the Dolbeault resolution. 
Let $X=\cup_{i=1}^rX_i$ be an open affine cover. 
Then for each $I\subset[1,\ldots,r]$, there exists by Noether normalization lemma 
a finite map $X_I\to\mathbb A_{\mathbb C}^{d_X}$. 
Hence the presheaves $(\pi^*(\mathcal A^{l,r,e}))_{|X_I}$ and $(\pi^*\Omega^{p,e})_{|X_I}$ are sheaves.
We have thus, by Poincare lemma on open balls, for each $I\subset[1,\ldots,r]$,
\begin{equation*}
H_{usu}^k(X_I^{an},\ad(\pi)(\pi^*\Omega^{p,e}))=H^k\Gamma(X_I^{an},\ad(\pi)(\pi^*\mathcal A^{\bullet\geq p,\bullet,e})),
\end{equation*}
indeed, since open balls form a basis of the complex topology, 
there exists $D^l$ such that $D^l\subset D^{l-1}\times\mathbb A^1$. We have
\begin{equation*}
\ad(\pi)(\pi^*\mathcal A^{\bullet\geq p,\bullet,e})(X_I^{an}):=
\varinjlim_{m:X_I\to Y\times\mathbb A^1}\mathcal A^e(m)(\mathcal A^{\bullet\geq p,\bullet,e}(Y\times\mathbb A^1))
\subset\mathcal A^{\bullet\geq p,\bullet}(X^{an})
\end{equation*}
Now, for each $I\subset[1,\ldots,r]$,
\begin{itemize}
\item if $q\circ m:X_I^{an}\xrightarrow{m}Y\times\mathbb A^1\xrightarrow{q}\mathbb A^1$ is non constant,
we can factor 
\begin{equation*}
m=(n\times I)\circ m':X_I^{an}\xrightarrow{m'}Y'\times\mathbb A^1\xrightarrow{n\times I}Y\times\mathbb A^1
\end{equation*}
with $\dim Y'=\dim X-1$ and $m'(X_I^{an})\subset Y'\times\mathbb A^1$ a dense open subset, and 
\begin{equation*}
\mathcal A^e(m'):(\mathcal A^{\bullet\geq p,\bullet,e}(Y'\times\mathbb A^1)\to\mathcal A^{\bullet\geq p,\bullet}(X_I^{an})
\end{equation*}
is injective, and 
\begin{eqnarray*}
H^k\mathcal A^e(m)(\mathcal A^{\bullet\geq p,\bullet,e}(Y\times\mathbb A^1))
&=&H^k\mathcal A^e(m')(\mathcal A^{\bullet\geq p,\bullet,e}(Y'\times\mathbb A^1)) \\
&=&\mathcal A^e(m')(H^k(\mathcal A^{\bullet\geq p,\bullet,e})(Y'\times\mathbb A^1))=0
\end{eqnarray*}
by Kunneth formula for $Y'\times\mathbb A^1$,
\item if $q\circ m:X\xrightarrow{m}Y\times\mathbb A^1\xrightarrow{q}\mathbb A^1$ is constant with value $c$,
\begin{equation*}
H^k\mathcal A^e(m)(\mathcal A^{\bullet\geq p,\bullet,e})(Y\times\mathbb A^1))=0 
\end{equation*}
since $c^*d\bar z=0$ and $c^*dz=0$, 
where $c:Y\times\left\{c\right\}\hookrightarrow Y\times\mathbb A^1$ is the closed embedding.
\end{itemize}
Hence, for each $I\subset[1,\ldots,r]$ and $k\in\mathbb Z$
\begin{equation*}
H^k(e)(H_{usu}^k(X_I^{an},\ad(\pi)(\pi^*\Omega^{p,e})))=0.
\end{equation*}
\end{proof}
By the spectral sequence associated to the bi complex $\Gamma(X_{\bullet},\mathcal A^{\bullet\geq p,\bullet,e})$, we get
\begin{equation*}
H^k(e)(H_{usu}^k(X^{an},\ad(\pi)(\pi^*\Omega^{p,e})))=0.
\end{equation*}
for each $k\in\mathbb Z$.

\begin{prop}\label{cGAGAlog}
Let $X\in\PSmVar(\mathbb C)$. Consider the morphism $\an_X:X^{an}\to X$ in $\RTop$ given by analytical functor.
\begin{itemize}
\item[(i)]The analytic De Rham cohomology class of an algebraic cycle is logarithmic and is of type $(d,d)$,
that is, for $Z\in\mathcal Z^d(X)$  
\begin{equation*}
[Z]:=H^{2d}\Omega(\gamma^{\vee}_Z)([Z])\subset H^{2d}OL_{X^{an},0}(H_{usu}^d(X^{an},\Omega^d_{X^{an},\log,0}))
\subset H_{DR}^{2d}(X^{an})=H^{2d}_{DR}(X).
\end{equation*}
\item[(ii)]Let $d\in\mathbb N$. Conversely, any $w\in H^{2d}OL_{X^{an},0}(H_{usu}^d(X^{an},\Omega^d_{X^{an},\log,0}))$ 
is the class of an algebraic cycle $Z\in\mathcal Z^d(X)\otimes\mathbb Q$, i.e. $w=[Z]$.
\item[(iii)] We have $H^jOL_{X^{an}}(H_{usu}^{j-l}(X^{an},\Omega^l_{X^{an},\log,0}))=0$ 
for $j,l\in\mathbb Z$ such that $2l>j$.
\end{itemize}
\end{prop}

\begin{proof}
Consider, for $j,l\in\mathbb Z$ and $X\in\SmVar(\mathbb C)$, 
\begin{eqnarray*}
L^{l,j-l}(X):=H^jOL_{X^{an},0}(H_{usu}^{j-l}(X^{an},\Omega^l_{X^{an},\log,0}))
:=H^jOL_0^{an}(H^{j-l}\Hom(\An^*\mathbb Z(X),E^{\bullet}_{usu}(\Omega_{\log,0}^{l,an}))).
\end{eqnarray*} 
Consider also for $j,l\in\mathbb Z$, $X\in\SmVar(\mathbb C)$ and $Z\subset X$ a closed subset,
\begin{eqnarray*}
L_Z^{l,j-l}(X):=H^jOL_{X^{an},0}(H_{usu,Z}^{j-l}(X^{an},\Omega_{X^{an},\log,0}^l))
:=H^jOL_0^{an}(H^{j-l}\Hom(\An^*\mathbb Z(X,X\backslash Z),E^{\bullet}_{usu}(\Omega_{\log,0}^{l,an}))).
\end{eqnarray*} 
Consider, for $j,l\in\mathbb Z$ and $X\in\SmVar(\mathbb C)$, 
\begin{eqnarray*}
L^{l,j-l,ne}(X):=H^j(\bar OL_0^{an})(H^{j-l}\Hom(\An^*\mathbb Z(X),E^{\bullet}_{usu}(\Omega_{\log,0}^{l,an,ne}))).
\end{eqnarray*}
Consider also for $j,l\in\mathbb Z$, $X\in\SmVar(\mathbb C)$ and $Z\subset X$ a closed subset,
\begin{eqnarray*}
L_Z^{l,j-l,ne}(X):=
H^j(\bar OL_0^{an})(H^{j-l}\Hom(\An^*\mathbb Z(X,X\backslash Z),E^{\bullet}_{usu}(\Omega_{\log,0}^{l,an,ne}))).
\end{eqnarray*}
We have, for $X\in\SmVar(\mathbb C)$ and $Z\subset X$ a smooth subvariety of (pure) codimension $d$, 
by proposition \ref{smXZ} and proposition \ref{UXusu}(i) an isomorphism
\begin{eqnarray*}
E_{usu}(\Omega^{l,an,ne}_{\log,0})(P_{Z,X}):
L_Z^{l,j-l,ne}(X)\xrightarrow{\sim}L_Z^{l,j-l,ne}(N_{Z/X})\xrightarrow{\sim}L^{l-d,j-l-d,ne}(Z).
\end{eqnarray*}

\noindent(i):Similar to the proof of \cite{B7} theorem 2.

\noindent(ii): Let $X\in\PSmVar(\mathbb C)$ and
\begin{eqnarray*}
\alpha=H^{2d}OL_{X^{an}}(\alpha)
\in L^{d,d}(X):&=&H^{2d}OL_{X^{an},0}(H_{usu}^d(X^{an},\Omega^d_{X^{an},\log,0})) \\ 
&=&H^{2d}OL_0^{an}(H^d\Hom(\An^*\mathbb Z(X),E^{\bullet}_{usu}(\Omega_{\log,0}^{d,an}))).
\end{eqnarray*}
Then,
\begin{eqnarray*}
H^{2d}(q)(\alpha)=H^{2d}(q)\circ H^{2d}OL_{X^{an},0}(\alpha)=H^{2d}\bar OL_0^{an}\circ H^{2d}(q)(\alpha) \\
\in L^{d,d,ne}(X):=H^{2d}\bar OL_0^{an}(H^d\Hom(\An^*\mathbb Z(X),E^{\bullet}_{usu}(\Omega_{\log,0}^{d,an,ne}))).
\end{eqnarray*}
Up to split $X$ by its connected components, we may assume that $X$ is connected. 
Let $j:U\hookrightarrow X$ be an open embedding such that $\Pic(U)=0$ (where $\Pic(U)$ is the Picard group of the algebraic variety $U$) 
and $X\backslash U:=D=D'\cup E$ where $D'\subset X$ is an ample divisor (add component if necessary) 
and $E\subset X$ is a smooth ample divisor (use Bertini for generic smoothness).
By proposition \ref{UXusu}(ii), we have 
\begin{equation*}
j^*H^{2d}(q)(\alpha)=0\in H_{usu}^d(U^{an},\Omega^{d,an,ne}_{\log,0})=0.
\end{equation*}
We get 
\begin{equation*}
H^{2d}(q)(\alpha)=H^{d}E^{\bullet}_{usu}(\Omega_{\log,0}^{d,an,ne})(\gamma^{\vee}_D)(\alpha'), \, 
\alpha'\in L_D^{l,j-l,ne}(X).
\end{equation*}
Denote $D^o\subset D$ its smooth locus, $l:X^o\hookrightarrow X$ a Zariski open subset such that $X^o\cap D=D^o$ and 
$n_D:\tilde D\xrightarrow{\epsilon}D\hookrightarrow X$, $\epsilon$ being a desingularization, 
and $l_D:D^o\hookrightarrow\tilde D$ the open embedding. 
We then have 
\begin{eqnarray*}
l^*\alpha'\in L^{d,d}_{D^o}(X^o)=L^{d-1,d-1}(D^o), \\ \alpha'=l_D^*\alpha'', \; 
\alpha''\in H_{usu}^{d-1}(\tilde D^{an},\Omega^{d-1,an,ne}_{\log,0}), \; \Omega(D(\mathbb Z(n_D)))(\alpha'')=\alpha+\beta', \; l^*\beta'=0
\end{eqnarray*}
We repeat this procedure with each connected components of $D^o$ instead of $X$.
By a finite induction of $d$ steps, we get 
\begin{equation*}
H^{2d}(q)(\alpha)=H^{d}E^{\bullet}_{usu}(\Omega_{\log,0}^{d,an,ne})(\gamma^{\vee}_Z)(\alpha'), \, 
\alpha'\in L_Z^{d,d,ne}(X)=L^{0,0,ne}(Z^o).
\end{equation*}
with $Z:=D_d\subset\cdots\subset D\subset X$ a pure codimension $d$ (Zariski) closed subset, thus 
\begin{equation*}
\alpha=\sum_in_i[Z_i]+\beta\in H^{2d}_{DR}(X), \; 
\beta\in H^{2d}(e)(\mathbb H_{usu}^{2d}(X^{an},\ad(\pi)(\pi^*\Omega^{\bullet,e})))
\end{equation*}
where $n_i\in\mathbb Q$ and $(Z_i)_{1\leq i\leq t}\subset Z$ are the irreducible components of $Z$.
But by lemma \ref{cGAGAloglem}, we have 
\begin{equation*}
H^k(e)(\mathbb H_{usu}^k(X^{an},\ad(\pi)(\pi^*\Omega^{\bullet,e})))=0. 
\end{equation*}
for each $k\in\mathbb Z$, $k\neq 0$. Hence $\beta=0$, that is
\begin{equation*}
\alpha=\sum_in_i[Z_i]\in H^{2d}_{DR}(X).
\end{equation*}

\noindent(iii): Let $j<2l$. Let $X\in\PSmVar(\mathbb C)$ and $w\in L^{l,j-l}(X)$. 
By proposition \ref{UXusu}(ii)', arguing as in the proof of (ii),
there exists $Z\subset X$ a closed subset of pure codimension $j-l$ such that
\begin{equation*}
H^j(q)(w)=H^{j-l}E^{\bullet}_{usu}(\Omega_{\log,0}^{l,an,ne})(\gamma^{\vee}_Z)(w'), \, w'\in L_Z^{l,j-l,ne}(X).
\end{equation*} 
For $Z'\subset X$ a closed subset of pure codimension $c$,
consider a desingularisation $\epsilon:\tilde Z'\to Z'$ of $Z'$ and denote $n:\tilde Z'\xrightarrow{\epsilon}Z'\subset X$.
The morphism in $\DA(k)$
\begin{equation*}
G_{Z',X}:M(X)\xrightarrow{D(\mathbb Z(n))}M(\tilde Z')(c)[2c]\xrightarrow{\mathbb Z(\epsilon)}M(Z')(c)[2c]
\end{equation*}
where $D:\Hom_{\DA(k)}(M_c(\tilde Z'),M_c(X))\xrightarrow{\sim}\Hom_{\DA(k)}(M(X),M(\tilde Z')(c)[2c])$
is the duality isomorphism from the six functor formalism (moving lemma of Suzlin and Voevodsky)
and $\mathbb Z(n):=\ad(n_!,n^!)(a_X^!\mathbb Z)$, is given by a morphism in $C(\SmVar(k))$
\begin{equation*}
\hat G_{Z',X}:\mathbb Z_{tr}(X)\to E_{et}(C_*\mathbb Z_{tr}(Z'))(c)[2c].
\end{equation*}
Let $l:X^o\hookrightarrow X$ be an open embedding such that $Z^o:=Z\cap X^o$ is the smooth locus of $Z$.
We have the following commutative diagram of abelian groups
\begin{equation*}
\xymatrix{0\ar[r] & H_{usu,Z}^{j-l}(X^{an},\Omega_{\log,0}^{l,an,ne})\ar[r]^{l^*} & 
H_{usu,Z^o}^{j-l}(X^{o,an},\Omega_{\log,0}^{l,an,ne})\ar[r]^{\partial} & 
H_{usu,Z\backslash Z^o}^{j-l+1}(X^{an},\Omega^{l,an,ne}_{\log,0})\ar[r] & \cdots \\
0\ar[r] & H^0_{usu}(Z^{an},\Omega^{2l-j,an,ne}_{\log,0})\ar[r]^{l*}\ar[u]^{\Omega(\hat G_{Z,X})} & 
H^0_{usu}(Z^{o,an},\Omega^{2l-j,an,ne}_{\log,0})\ar[r]^{\partial}\ar[u]^{\Omega(P_{Z^o,X^o})} &
H_{usu,Z\backslash Z^o}^1(Z^{an},\Omega^{2l-j,an,ne}_{\log,0})\ar[u]^{\Omega(\hat G_{Z,X})}\ar[r] & \cdots}
\end{equation*}
whose rows are exact sequences. Consider 
\begin{equation*}
l^*w'=\Omega(P_{Z^o,X^o})(w^o)\in H_{usu,Z^o}^{j-l}(X^{o,an},\Omega_{\log,0}^{l,an,ne}), \; 
w^o\in H^0_{et}(Z^{o,an},\Omega^{2l-j,an,ne}_{\log,0}).
\end{equation*}
Since $\partial l^*w'=0\in H_{usu,Z\backslash Z^o}^{j-l+1}(X^{an},\Omega^{l,an,ne}_{\log,0})$,
we get, considering the smooth locus
$(Z\backslash Z^o)^o\subset Z\backslash Z^o$ and $X^{oo}\subset X$ is an open subset
such that $X\cap(Z\backslash Z^o)^o=X^{oo}\cap (Z\backslash Z^o)$, 
\begin{equation*}
\partial w^o=0\in H_{usu,Z\backslash Z^o}^1(Z^{an},\Omega^{2l-j,an,ne}_{\log,0}), 
\end{equation*}
since $H_{usu,(Z\backslash Z^o)\backslash(Z\backslash Z^o)^o}^1(Z^{an},\Omega^{2l-j,an,ne}_{\log,0})=0$
for dimension reasons, that is 
\begin{equation*}
H^j(q)(w)=\Omega(\hat G_{Z,X})(w'), \, \mbox{with} \, w'\in H^0_{et}(Z^{an},\Omega^{2l-j,an,ne}_{\log,0}).
\end{equation*} 
Hence $H^j(q)(w)=0$ since $H^0(Z^{'an},\Omega^k_{\log,0})=0$ for all $k>0$ and all $Z'\in\PVar(\mathbb C)$.
But by lemma \ref{cGAGAloglem}
\begin{equation*}
H(e)(\mathbb H_{usu}^j(X^{an},\ad(\pi)(\pi^*\Omega^{\bullet,e})))=0. 
\end{equation*}
Hence $w=0$.
\end{proof}

\section{Hypersurface De Rham cohomology and its complex periods}

We denote $\mathbb P^N:=\mathbb P^N_{\mathbb C}:=\Proj(\mathbb C[z_0,\cdots,z_N])$.
Let $X=V(f)\subset\mathbb P^N$ be a smooth hypersurface. If $N-1=2p$ is even, we consider the decompositions
\begin{equation*}
H_{DR}^{N-1}(X)=\mathbb C[\Lambda\cap X]\oplus H_{DR,v}^{N-1}(X) \; , \; 
H_{\sing}^{N-1}(X^{an},\mathbb Q)=\mathbb C[\Lambda\cap X]\oplus H_{\sing,v}^{N-1}(X^{an},\mathbb Q)
\end{equation*}
of vector spaces, where $\Lambda\subset\mathbb P^N$ is a linear subspace of codimension $p$.
If $N-1=2p+1$ is even, we set
\begin{equation*}
H_{DR,v}^{N-1}(X):=H_{DR}^{N-1}(X) \; , \; H_{\sing,v}^{N-1}(X^{an},\mathbb Q):=H_{\sing}^{N-1}(X^{an},\mathbb Q).
\end{equation*}

Let $X=V(f)\subset\mathbb P^N$ be a smooth hypersurface. 
Denote $j:U:=\mathbb P^N\backslash X\hookrightarrow\mathbb P^N$ the open embedding. We consider the algebraic Hodge module 
\begin{equation*}
j_{*Hdg}(O_U,F_b):=((j_*O_U,F),j_*\mathbb Q_{U^{an}},j_*\alpha(U))\in HM(\mathbb P^N), 
\end{equation*}
where $F^kj_*O_U=\mathbb C[x_1,\cdots x_N][\frac{1}{f^k}]$. Recall for each $p\in\mathbb Z$,
\begin{equation*}
F^pH^N_{DR}(U):=H^N\iota^{\geq p}(\mathbb H^N(\mathbb P^N,\Omega^{\bullet\geq p}_{\mathbb P^N}(\log X)))
=H^N\iota^{\geq p}H^N(\mathbb P^N,\Omega^{\bullet}_{\mathbb P^N}\otimes_{O_{\mathbb P^N}}F^{\bullet-p}j_*O_U)
\end{equation*}
and
\begin{equation*}
F^pH^{N-1}(X):=H^{N-1}\iota^{\geq p}(\mathbb H^{N-1}(X,\Omega_X^{\bullet\geq p})).
\end{equation*}
where
\begin{equation*}
\iota^{\geq p}:\Omega^{\bullet\geq p}_{\mathbb P^N}(\log X)\hookrightarrow\Omega^{\bullet}_{\mathbb P^N}(\log X) \; , \;
\iota^{\geq p}:\Omega^{\bullet}_{\mathbb P^N}\otimes_{O_{\mathbb P^N}}F^{\bullet-p}j_*O_U\hookrightarrow
\Omega^{\bullet}_{\mathbb P^N}\otimes_{O_{\mathbb P^N}}j_*O_U=j_*\Omega^{\bullet}_U
\end{equation*}
and $\iota^{\geq p}:\Omega^{\bullet\geq p}_X\hookrightarrow\Omega^{\bullet}_X$
are the canonical embeddings.
Then the map $Res_{X,\mathbb P^N}$ of definition \ref{ResMot} induces for each $p\in\mathbb Z$, canonical isomorphisms 
\begin{equation*}
H^N\Omega(Res_{X,\mathbb P^N}):F^{p+1}H_{DR}^N(U)\xrightarrow{\sim}F^pH_{DR}^{N-1}(X)_v \; , \;
H^N\Bti(Res_{X,\mathbb P^N}):H_{\sing}^N(U,\mathbb Q)\xrightarrow{\sim}H_{\sing}^{N-1}(X,\mathbb Q)_v.
\end{equation*}

\begin{lem}\label{OBac}
Let $X=V(f)\subset\mathbb P^N$ be a smooth hypersurface. 
Denote $j:U:=\mathbb P^N\backslash X\hookrightarrow\mathbb P^N$ the open embedding. 
Let $B\subset\mathbb P^N$, $\phi:B\simeq D(0,1)^N$ be an open ball 
such that $\phi(X\cap B)=\left\{0\right\}\times D(0,1)^{N-1}$. Then,
\begin{itemize}
\item[(i)] For $q\in\mathbb Z$, $q\neq 0$,
$\mathbb H_{usu}^q(B,(\Omega^{\bullet}_{\mathbb P^N}\otimes_{O_{\mathbb P^N}}F^{\bullet-p}j_*O_U)^{an})=0$.
\item[(ii)]For $q\in\mathbb Z$, $q\neq 0$,
$H_{usu}^q(B,\Omega^{N-p}_{\mathbb P^{N,an}}(\log X)^{\partial=0})=0$.
\item[(iii)]For $q\in\mathbb Z$, $q\neq 0$,
$H_{usu}^q(B\backslash X,\Omega^{N-p}_{\mathbb P^{N,an},\log,0})=0$.
\end{itemize}
\end{lem}

\begin{proof}
\noindent(i): Since $B$ is Stein, we have for $r,q\in\mathbb Z$, $q\neq 0$, 
\begin{equation*}
H_{usu}^q(B,(\Omega^r_{\mathbb P^N}\otimes_{O_{\mathbb P^N}}F^{r-p}j_*O_U)^{an})=0.
\end{equation*}
On the other hand the De Rham complex of $(j_*O_{B\backslash X},F)$ is acyclic (see e.g. \cite{Voisin} corollary 18.7)

\noindent(ii): Consider a sequence of sub-balls 
\begin{equation*}
l_n:B_n:=D(0,r_n)^N\hookrightarrow B:=D(0,1)^N, n\in\mathbb N, \; \mbox{with} \; 
r_n\to 0 \; \mbox{when} \; n\to\infty.
\end{equation*}
Since $(B_n:=D(0,r_n))_{n\in\mathbb N}$ form a basis of neighborhood of $0\in B$ and 
since $\Omega^{N-p}_{\mathbb P^{N,an}}(\log X)^{\partial=0}$ is a single presheaf, 
there exists $n:=n_{\alpha}\in\mathbb N$ (depending on $\alpha$) such that 
\begin{equation*}
l_n^*\alpha=0\in H^q(B_n,\Omega^{N-p}_{\mathbb P^{N,an}}(\log X)^{\partial=0}). 
\end{equation*}
But it follows immediately from the definition of the the presheaf 
$\Omega^{N-p}_{\mathbb P^{N,an}}(\log X)^{\partial=0}$ considering the isomorphism 
\begin{equation*}
\phi_n:B_n\xrightarrow{\sim}B, \; \phi_n(z):=r_n^{-1}z
\end{equation*}
so that $(\phi_n)^{-1}\circ l_n=r_nI$, that 
\begin{equation*}
l_n^*\alpha=l_n^*\phi_n^{-1*}\phi_n^*\alpha=r_n\phi_n^*\alpha
\in H^q(B_n,\Omega^{N-p}_{\mathbb P^{N,an}}(\log X)^{\partial=0}).
\end{equation*}
Hence $\phi_n^*\alpha=0$. Since $\phi_n$ is an isomorphism, we get $\alpha=0$.

\noindent(iii):We show by induction on $m\in\mathbb N$, $m\geq 1$, that $H_{usu}^q(B\backslash X,\wedge^mO^*_{\mathbb P^{N,an}})=0$.
We have by the exponential exact sequence
\begin{equation*}
H_{usu}^q(B\backslash X,O^*_{\mathbb P^{N,an}})=H_{usu}^{q+1}(B\backslash X,\mathbb Z)=0,
\end{equation*}
since $B\backslash X\simeq D(0,1)^*\times D(0,1)^{N-1}$ and $q+1>1$.
This implies that we have also by induction on $m'\in\mathbb N$, using the exponential exact sequence 
\begin{equation*}
H_{usu}^q(B\backslash X,O^*_{\mathbb P^{N,an}}\otimes O_{\mathbb P^{N,an}}^{\otimes m'})=0, \; \mbox{for \; all} \; m'\in\mathbb N
\end{equation*}
Assume $H_{usu}^q(B\backslash X,\wedge^{m-1}O^*_{\mathbb P^{N,an}}\otimes O_{\mathbb P^{N,an}}^{\otimes m'})=0$ for all $m'\in\mathbb N$. 
By the exponential exact sequence
\begin{equation*}
0\to\wedge^{m-1}O^*_{\mathbb P^{N,an}}\otimes O_{\mathbb P^{N,an}}^{\otimes m'}\to\wedge^{m-1}O^*_{\mathbb P^{N,an}}\otimes O_{\mathbb P^{N,an}}^{\otimes m'+1}
\xrightarrow{I\otimes e}\wedge^mO^*_{\mathbb P^{N,an}}\otimes O_{\mathbb P^{N,an}}^{\otimes m'}\to 0
\end{equation*}
and the induction hypothesis we get
\begin{equation*}
H_{usu}^q(B\backslash X,\wedge^mO^*_{\mathbb P^{N,an}}\otimes O_{\mathbb P^{N,an}}^{\otimes m'})=0, \; \mbox{for \; all} \; m'\in\mathbb N.
\end{equation*}
In particular $H_{usu}^q(B\backslash X,\wedge^mO^*_{\mathbb P^{N,an}})=0$, this proves (iii).

\end{proof}

We have by the local acyclicity of the De Rham complex of $(j_*O_U,F)$ a canonical map $B_{\partial}$ :

\begin{defi}\label{HUDR}
Let $X=V(f)\subset\mathbb P^N$ be a smooth hypersurface. 
Denote $j:U:=\mathbb P^N\backslash X\hookrightarrow\mathbb P^N$ the open embedding.
Consider the canonical embedding
\begin{equation*}
\iota_{\partial}:\Omega^{N-p}_{\mathbb P^{N,an}}(\log X)^{\partial=0}[-p]\hookrightarrow
(\Omega^{\bullet}_{\mathbb P^N}\otimes_{O_{\mathbb P^N}}F^{\bullet-p}j_*O_U)^{an}
\end{equation*}
Let $\mathbb P^{N,an}=\cup_{i=1}^nB_i$ an open cover, with $\phi_i:B_i\simeq D(0,1)^N$ open balls
such that $\phi_i(X\cap B_i)=\left\{0\right\}\times D(0,1)^{N-1}$. 
Denote, for $I\subset[1,\ldots,n]$, $l_I:B_I\hookrightarrow\mathbb P^N$ the open embeddings, with $B_I:=\cap_{i\in I}B_i$. 
Since, for each $I\subset[1,\ldots,n]$ and $p\in\mathbb Z$, 
\begin{equation*}
l_I^*\iota_{\partial}:l_I^*\Omega^{N-p}_{\mathbb P^{N,an}}(\log X)^{\partial=0}[-p]\hookrightarrow
l_I^*(\Omega^{\bullet}_{\mathbb P^N}\otimes_{O_{\mathbb P^N}}F^{\bullet-p}j_*O_U)^{an}
\end{equation*}
is an quasi-isomorphism (see e.g. \cite{Voisin} corollary 18.7), 
the spectral sequence of the filtration on the total complex of a bi-complex gives, for each $p\in\mathbb Z$, 
a canonical isomorphism, where the first and last map are isomorphism by lemma \ref{OBac}(i) and (ii) respectively,
\begin{eqnarray*}
B_{\partial}:F^pH_{DR}(U):=H^N\iota^{\geq p}
\mathbb H_{zar}^N(\mathbb P^N,\Omega^{\bullet}_{\mathbb P^N}\otimes_{O_{\mathbb P^N}}F^{\bullet-p}j_*O_U) \\ 
\xrightarrow{H^N((l_i^*=\Omega(l_i))_{1\leq i\leq n})}
H^N\iota^{\geq p}H^N\Gamma(B_{\bullet},(\Omega^{\bullet}_{\mathbb P^N}\otimes_{O_{\mathbb P^N}}F^{\bullet-p}j_*O_U)^{an}) \\
\to H^p\Gamma(B_{\bullet},\Omega^{N-p}_{\mathbb P^{N,an}}(\log X)^{\partial=0})
\xrightarrow{H^N((l_i^*=\Omega(l_i))_{1\leq i\leq n})^{-1}}
H^p_{usu}(\mathbb P^N,\Omega^{N-p}_{\mathbb P^{N,an}}(\log X)^{\partial=0}),
\end{eqnarray*}
given by, for $w\in\Gamma(U,\Omega^N_{\mathbb P^N})$, inductively
by a $p$ step induction where $p$ is the order of the pole of $w$ at $X$,
\begin{itemize}
\item taking for each $1\leq i\leq n$, $l_i^*w=\partial w_{1i}$
since $\partial w=0$ for dimensional reason and by exactness of $l_i^*\iota_{\partial}$, 
\item we then have
for each $1\leq i,j\leq n$, considering $w_{1ij}:=l_{ij}^*(w_{1i}-w_{1j})$, 
$w_{1ij}=\partial w_{2ij}$ since 
\begin{equation*}
\partial w_{1ij}=l_{ij}^*(\partial w_{1i}-\partial w_{1j})=l_{ij}^*(l_i^*w-l_j^*w)=0
\end{equation*}
and by exactness of $l_{ij}^*\iota_{\partial}$, 
\item at the final $p$ step, 
$B(w):=w_{pI}\in H^N(B_{\bullet},\Omega^{N-p}_{\mathbb P^{N,an}}(\log X)^{\partial=0})$, $card I=p$.
\end{itemize}
By definition $B_{\partial}(w)=(H^p\iota_{\partial})^{-1}(w)$.
\end{defi}

Let $X=V(f)\subset\mathbb P^N$ be a smooth hypersurface. Denote $U:=\mathbb P^N\backslash X$. For $l\in\mathbb Z$,
we consider the factorization in $C(U^{an})$
\begin{equation*}
OL^l_{U^{an},0}:\Omega^l_{U^{an},\log,0}[-l]\xrightarrow{\widehat{OL_{U^{an},0}^l}[-l]}
\Omega^l_{\mathbb P^{N,an}}(\log X)^{\partial=0}[-l]
\xrightarrow{\iota_{\partial}}(\Omega^{\bullet}_{\mathbb P^N}\otimes_{O_{\mathbb P^N}}F^{\bullet-p}j_*O_U)^{an}
\xrightarrow{\iota^{\geq l}}\Omega^{\bullet}_{U^{an}},
\end{equation*}
where 
\begin{equation*}
\iota_{\partial}:\Omega^l_{\mathbb P^{N,an}}(\log X)^{\partial=0}[-l]\hookrightarrow
(\Omega^{\bullet}_{\mathbb P^N}\otimes_{O_{\mathbb P^N}}F^{\bullet-p}j_*O_U)^{an}
\end{equation*}
is an equivalence usu local (see \cite{Voisin}), and the factorization in $C(U^{an})$
\begin{eqnarray*}
OL^l_{U^{an},0}:\Omega^l_{U^{an},\log,0,v}[-l]:=\Omega^l_{U^{an},\log,0}/j^*\Omega^l_{\mathbb P^{N,an},\log,0}[-l]
\xrightarrow{\widehat{OL_{U^{an},0}^l}[-l]} \\
\Omega^l_{\mathbb P^{N,an}}(\log X)^{\partial=0}/j^*\Omega^l_{\mathbb P^{N,an}}[-l]
\xrightarrow{\iota^{\geq l}\circ\iota_{\partial}}\Omega^{\bullet}_{U^{an}}/j^*\Omega^l_{\mathbb P^{N,an}}.
\end{eqnarray*}

We will use the following lemma

\begin{lem}\label{Dlem}
Let $X=V(f)\subset\mathbb P^N$, $N=2p+1$, be a smooth projective hypersurface. Denote $U:=\mathbb P^N\backslash X$.
Let $\mathbb P^{N,an}=\cup_{i=1}^nB_i$ an open cover, with $\phi_i:B_i=B(c_i,r_i)\simeq D(0,1)^N$ open balls such that 
\begin{equation*}
\phi_i(X\cap B_i)=\left\{0\right\}\times D(0,1)^{N-1} \; \mbox{ and } \; \phi_i(z):=\phi_i(z_0,z_1,\ldots,z_N)=(z_0,\ldots,f(z),\ldots,z_N). 
\end{equation*}
Note that since $\mathbb P^{N,an}$ is compact, up to take finite subcover, we can assume that this cover is finite. 
Denote, for $I\subset[1,\ldots,n]$, $B_I:=\cap_{i\in I}B_i$ and consider $B(c_I,r_I)\subset B_I$ a ball of maximal radius contained in $B_I$.
For $I\in(n,q)$, $\nu\in(N,r)$ and $n=(n_1,\cdots,n_N)\in\mathbb N^N$, we denote, for $\gamma_I\in C^{\diff}_{r+2}(B_I\backslash X^{an},\mathbb Z)$
\begin{equation*}
P_{n,\nu}(\gamma_I):=\int_{\gamma_I}d(z_1-c_1)^{n_1}\cdots (z_N-c_N)^{n_N}\wedge dz_{\nu_1}\wedge\cdots\wedge dz_{\nu_{p-1}}\wedge df/f, \; c_I:=(c_1,\cdots,c_N).
\end{equation*}
Then, we have  
\begin{itemize}
\item[(i)]For $\gamma_I\in C^{\diff}_{r+2}(B_I\backslash X^{an},\mathbb Z)$ such that 
$\partial\gamma_I\in\bigoplus_{J\supset I}C^{\diff}_{r+1}(B_J\backslash X^{an},\mathbb Z)$ and 
\begin{eqnarray*}
P_{n,\nu}(\gamma_I)=\sum_{J\supset I}d_{n,\nu,J}, \; \mbox{for \,all} \; n\in\mathbb N^N, \nu\in(N,r) \\ 
(d_{n,\nu,J})\in\mathbb C^{\mathbb N^N\times(n,q+1)\times(N,r)}, \, |d_{n,\nu,J}|\leq (r_J-|c_I-c_J|)^n \; \mbox{for \,each} \; J\supset I,
\end{eqnarray*}
there exists $(\gamma'_{I,J})\in\bigoplus_{J\supset I} C^{\diff}_{r+2}(B_J\backslash X^{an},\mathbb Z)$ such that
\begin{equation*}
P_{n,\nu}(\gamma_I)=\sum_{J\supset I}P_{n,\nu}(\gamma'_{I,J}) \mbox{for \,all} \; n\in\mathbb N^N, \nu\in(N,r).
\end{equation*}
\item[(ii)]For $\gamma=(\gamma_I)\in C^{\diff}_{\bullet}(B_{\bullet}\backslash X^{an},\mathbb Q)^{\partial_{\bullet}+\partial=0}$ 
such that $ev(U)(\xi)(\gamma)=0$ for all 
\begin{equation*}
\xi=(\xi_I)\in\Gamma(B_{\bullet}\backslash X^{an},\Omega^{p+1}_{\mathbb P^N,\log,0}\otimes\mathbb C)^{\partial_{\bullet}=0},
\end{equation*}
there exists $\gamma'=(\gamma'_J)\in C^{\diff}_{\bullet}(B_{\bullet}\backslash X^{an},\mathbb Q)$ 
such that $ev(U)(w)(\gamma)=ev(U)(w)(\partial_{\bullet}\gamma')$ for all $w=(w_I)\in\Gamma(B_{\bullet},\Omega_{\mathbb P^N}^{p+1,\partial=0}(\log X))$.
\end{itemize}
\end{lem}

\begin{proof}
\noindent(i)For $I\in(n,q)$, $\nu\in(N,r)$ and $n=(n_1,\cdots,n_N)\in\mathbb N^N$, we denote, for $\eta_I\in C^{\diff}_{r+1}(B_I\backslash X^{an},\mathbb Z)$
\begin{equation*}
Q_{n,\nu}(\eta_I):=\int_{\eta_I}(z_1-c_1)^{n_1}\cdots (z_N-c_N)^{n_N}\wedge dz_{\nu_1}\wedge\cdots\wedge dz_{\nu_{p-1}}\wedge df/f, \; c_I:=(c_1,\cdots,c_N)
\end{equation*}
We have then $P_{n,\nu}(\gamma_I)=Q_{n,\nu}(\partial\gamma_I)$.
As $\gamma_I$ have integer coefficients, we consider the splitting
\begin{equation*}
\partial\gamma_I=\sum_{k=1}^e\eta_k^+-\sum_{k=1}^e\eta_k^-, \eta_k^+,\eta_k^-\in\Hom_{\Diff}(\mathbb I^{r+1},B_I\backslash X^{an}),
\end{equation*} 
where $\mathbb I^m:=[0,1]^m\subset\mathbb R^m$ (we use the cubical version instead of the triangle version of simplicial realization).
We have then
\begin{equation*}
P_{n,\nu}(\gamma_I)=\sum_{k=1}^eQ_{n,\nu}(\eta_k^+)-\sum_{k=1}^eQ_{n,\nu}(\eta_k^-), 
\end{equation*} 
Now we split $\cup_{J\supset I}B_J$ into connected components : $\cup_{J\supset I}B_J=\sqcup_{l=1}^s(\cup_{J\supset I_l}B_J)$.
Since 
\begin{eqnarray*}
P_{n,\nu}(\gamma_I)=\sum_{J\supset I}d_{n,\nu,J}, \; \mbox{for \,all} \; n\in\mathbb N^N, \nu\in(N,r) \\ 
(d_{n,\nu,J})\in\mathbb C^{\mathbb N^N\times(n,q+1)\times(N,r)}, \, |d_{n,\nu,J}|\leq (r_J-|c_I-c_J|)^n \; \mbox{for \,each} \; J\supset I 
\end{eqnarray*}
we have
\begin{equation*}
card(k\in[1,\cdots,e], \eta_k^+\in\cup_{J\supset I_l}B_J)=card(k\in[1,\cdots,e], \eta_k^-\in\cup_{J\supset I_l}B_J)=q_l.
\end{equation*} 
Indeed otherwise we have
\begin{equation*}
|P_{n,\nu}(\gamma_I)|>\sum_{J\subset I}(r_J-|c_I-c_J|)^n 
\end{equation*} 
since then there exists $J',J''\supset I$ such that $|c_{J'}-c_{J''}|>r_{J'}+r_{J''}$ contradiction.
Hence, for each $1\leq l\leq s$ and $1\leq k\leq q_l$, there exist $\gamma'_{I,J,k}\in\Hom_{\Diff}(\mathbb I^{r+2},B_J\backslash X^{an})$, $J\supset I_l$ such that 
\begin{equation*}
\eta_k^+-\eta_k^-=\sum_{J\supset I_l}\partial\gamma'_{I,J,k}
\end{equation*}
Denote for $J\supset I_l$, $\gamma'_{I,J}:=\sum_{k=1}^{q_l}\gamma'_{I,J,k}$. We have then
\begin{equation*}
\partial\gamma_I=\sum_{l=1}^s\sum_{k=1}^{q_l}\eta_k^+-\eta_k^-=
\sum_{l=1}^s\sum_{k=1}^{q_l}\sum_{J\supset I_l}\partial\gamma'_{I,J,k}=\sum_{J\supset I}\partial\gamma'_J.
\end{equation*} 
Hence,
\begin{equation*}
P_{n,\nu}(\gamma_I)=Q_{n,\nu}(\partial\gamma_I)=\sum_{J\supset I}Q_{n,\nu}(\partial\gamma'_{I,J})=\sum_{J\supset I}P_{n,\nu}(\gamma'_{I,J}),
\mbox{for \,all} \; n\in\mathbb N^N, \nu\in(N,r).
\end{equation*}

\noindent(ii):Up to clear denominators, that is replace $\gamma$ by $m\gamma$ with some $m\in\mathbb N$, 
we may assume that $\gamma=(\gamma_I)\in C^{\diff}_{\bullet}(B_{\bullet}\backslash X^{an},\mathbb Z)^{\partial_{\bullet}+\partial=0}$.
So let  $\gamma=(\gamma_I)\in C^{\diff}_{\bullet}(B_{\bullet}\backslash X^{an},\mathbb Z)^{\partial_{\bullet}+\partial=0}$
such that $ev(U)(\xi)(\gamma)=0$ for all $\xi=(\xi_I)\in\Gamma(B_{\bullet}\backslash X,\Omega^{p+1}_{\mathbb P^N,\log,0}\otimes\mathbb C)^{\partial_{\bullet}=0}$.
For 
\begin{equation*}
w=(w_I)=(\sum_{\nu\in(2p+1,p-1)}dh_{I,\nu}\wedge dz_{\nu_1}\wedge\cdots\wedge dz_{\nu_{p-1}})\wedge df/f\in\Gamma(B_{\bullet},\Omega_{\mathbb P^N}^{p+1,\partial=0}(\log X), 
\end{equation*}
where $h_{I,\nu}=\sum_{n=(n_1,\cdots,n_N)\in\mathbb N^N}a_{n,\nu,I}(z_1-c_1)^{n_1}\cdots (z_N-c_N)^{n_N}\in O(B_I)$, we have 
\begin{eqnarray*}
ev(U)(w)(\gamma)= \\
\sum_{n=(n_1,\cdots,n_N)\in\mathbb N^N}\sum_{\nu\in(2p+1,p-1)}\sum_{I\in(n,p)}a_{n,\nu,I}
\int_{\gamma_I}d(z_1-c_1)^{n_1}\cdots (z_N-c_N)^{n_N}\wedge dz_{\nu_1}\wedge\cdots\wedge dz_{\nu_{p-1}}\wedge df/f \\
=\sum_{n=(n_1,\cdots,n_N)\in\mathbb N^N}\sum_{\nu\in(2p+1,p-1)}\sum_{I\in(n,p)}a_{n,\nu,I}P_{n,\nu}(\gamma_I). 
\end{eqnarray*}
Note that  $\gamma_I\in C^{\diff}_{p+1}(B_I\backslash X^{an},\mathbb Z)$ is not closed.
By hypothesis if $w=\xi+\partial_{\bullet}\eta$ with $\xi=(\xi_I)\in\Gamma(B_{\bullet}\backslash X,\Omega^{p+1}_{\mathbb P^N,\log,0}\otimes\mathbb C)^{\partial_{\bullet}=0}$, that is if 
\begin{equation*}
exp((\partial_{\bullet}(h_{I,\nu}))_J)=f^q+(\partial_{\bullet}(\eta))_{J,\nu} \; \mbox{or equivalently} \; 
l_{J,\nu}:=(\partial_{\bullet}(h_{I,\nu}))_J=log(f^q+(\partial_{\bullet}(\eta))_{J,\nu})
\end{equation*}
for some $q\in\mathbb N$ and each $J\in(n,p+1)$, $\nu\in(2p+1,p-1)$, we have $ev(U)(w)(\gamma)=0$.
Equivalently there exists 
\begin{eqnarray*}
(d_{n,\nu,J})\in\mathbb C^{\mathbb N^N\times(2p+1,p-1)\times(n,p+1)}, \, |d_{n,\nu,J}|\leq (r_J-|c_I-c_J|)^n \; \mbox{for \,each} \; J\supset I,  
\end{eqnarray*}
such that 
\begin{equation*}
\sum_{n\in\mathbb N^N}\sum_{\nu\in(2p+1,p-1)}\sum_{I\in(n,p)}a_{n,\nu,I}P_{n,\nu}(\gamma_I)=
\sum_{n\in\mathbb N^N}\sum_{\nu\in(2p+1,p-1)}\sum_{J\in(n,p+1)}(\partial_{\bullet}a_{n,\nu,I})_Jd_{n,\nu,J}
\end{equation*}
for all $(a_{n,\nu,I})\in\mathbb C<z_1,\cdots,z_N>^{(2p+1,p-1)\times(n,p)}\subset\mathbb C^{\mathbb N^N\times(2p+1,p-1)\times(n,p)}$ convergent of radius at least $r_I$  
(that is the map $ev(U)(\gamma)$ through $\partial_{\bullet}$).
Note that for $J\supset I$, $|z-c_J|\leq r_J$ implies that $|z-c_I|\leq r_J-|c_I-c_J|$, 
and that $\sum_{n\in\mathbb N^N}(\partial_{\bullet}a_{n,\nu,I})_Jz^n$ converge if and only if $|z-c_J|\leq r_J$, hence $|d_{n,\nu,J}|\leq (r_J-|c_I-c_J|)^n$
since there exists $(a_{n,\nu,I})\in\mathbb C<z_1,\cdots,z_N>^{(2p+1,p-1)\times(n,p)}$ such that $(\partial_{\bullet}a_{n,\nu,I})_J$ is of radius exactly $r_J$. 
Hence
\begin{equation*}
\sum_{n\in\mathbb N^N}\sum_{\nu\in(2p+1,p-1)}\sum_{I\in(n,p)}a_{n,\nu,I}(P_{n,\nu}(\gamma_I)-(\partial_{\bullet}d_{n,\nu,J})_I)=0
\end{equation*}
for all $(a_{n,\nu,J})\in\mathbb C<z_1,\cdots,z_N>^{(2p+1,p-1)\times(n,p)}$, that is 
\begin{eqnarray*}
P_{n,\nu}(\gamma_I)=(\partial_{\bullet}d_{n,\nu,J})_I, \, |d_{n,\nu,J}|\leq (r_J-|c_I-c_J|)^n \; \mbox{for \,each} \; J\supset I \\
\mbox{for \,all} \; n\in\mathbb N^N,\nu\in(2p+1,p-1),I\in(n,p).
\end{eqnarray*}
Up to replace $(\gamma_I)$ by $(\gamma_I)+\partial_{\bullet}(\gamma''_J)+(\partial\gamma'_I)$ for some $(\gamma''_J)$ and some $(\gamma'_I)$, 
we may consider the normalized chain complex of $C^{\diff}_{p+1}(B_{\bullet}\backslash X^{an},\mathbb Z)/\partial$ and assume that 
\begin{equation*}
(\gamma_I)\in N(C^{\diff}_{p+1}(B_{\bullet}\backslash X^{an},\mathbb Z)/\partial). 
\end{equation*}
Then by (i) there exists for each $I\in(n,p)$ $(\gamma'_{I,J})\in\bigoplus_{J\supset I} C^{\diff}_{p+1}(B_J\backslash X^{an},\mathbb Z)$ such that
\begin{equation*}
P_{n,\nu}(\gamma_I)=P_{n,\nu}(\sum_{J\supset I}\gamma'_{I,J}) \;  \mbox{for \,all} \; n\in\mathbb N^N,\nu\in(2p+1,p-1).
\end{equation*} 
Now, consider
\begin{equation*}
\gamma':=\sum_{I\in(n,p)}(\gamma'_{I,J})=(\gamma'_J)\in C^{\diff}_{p+1}(B_{\bullet}\backslash X^{an},\mathbb Z), \; 
\gamma'_J:=\sum_{I\in(n,p)}\gamma'_{I,J}
\end{equation*}
Note that of course $\gamma\neq\partial_{\bullet}\gamma'$ since $[\gamma]\in H^N(U^{an},\mathbb Q)$ is a non trivial class.
This gives
\begin{eqnarray*}
ev(U)(w)(\gamma)&=&\sum_{n\in\mathbb N^N}\sum_{\nu\in(2p+1,p-1)}\sum_{I\in(n,p)}a_{n,\nu,I}P_{n,\nu}(\gamma_I) \\
&=&\sum_{n\in\mathbb N^N}\sum_{\nu\in(2p+1,p-1)}\sum_{I\in(n,p)}a_{n,\nu,I}P_{n,\nu}(\sum_{J\supset I}\gamma'_{I,J})=ev(w)(\partial_{\bullet}\gamma').
\end{eqnarray*}

\end{proof}

The main result of this section is the following :

\begin{prop}\label{HUQ}
Let $X=V(f)\subset\mathbb P^N$, $N=2p+1$, be a smooth projective hypersurface. Denote $U:=\mathbb P^N\backslash X$.
We have 
\begin{equation*}
F^{p+1}H^N(U^{an},\mathbb Q):=ev(U)(F^{p+1}H^N_{DR}(U))\cap H^N_{\sing}(U^{an},\mathbb Q)=OL^{N}_{U^{an},0}(H_{usu}^p(U^{an},\Omega_{U^{an},\log,0}^{N-p})).
\end{equation*}
That is for each $w\in F^{p+1}H^N_{DR}(U)$, the following assertions are equivalent
\begin{itemize}
\item[(i)] $ev(U)(w)\in H_{\sing}^N(U^{an},2i\pi\mathbb Q)$,
\item[(ii)] $w\in OL^{N}_{U^{an},0}(H_{usu}^p(U^{an},\Omega_{U^{an},\log,0}^{N-p}))$.
\end{itemize}
\end{prop}

\begin{proof}
\noindent We first show that 
\begin{equation*}
F^{p+1}H^N(U^{an},\mathbb Q)\supset OL^{N}_{U^{an},0}(H_{usu}^p(U^{an},\Omega_{U^{an},\log,0}^{N-p})), 
\end{equation*}
that is for each $w\in F^{p+1}H^N_{DR}(U)$ (ii) implies (i): 
Follows from proposition \ref{cGAGAlog}(ii) for $X$ by the map $Res_{X,\mathbb P^N}:M(X)\to M(U)[-1]$ given in definition \ref{ResMot}.

\noindent We now prove that 
\begin{equation*}
F^{p+1}H^N(U^{an},\mathbb Q)=OL^{N}_{U^{an},0}(H_{usu}^p(U^{an},\Omega_{U^{an},\log,0}^{N-p})), 
\end{equation*}
that is for each $w\in F^{p+1}H^N_{DR}(U)$ (i) implies (ii): We will show that 
\begin{eqnarray*}
ev(U)(F^{p+1}H^N_{DR}(U))^{\perp}\cap H_N^{\sing}(U^{an},\mathbb Q)\supset 
ev(U)(OL^{N}_{U^{an},0}(H_{usu}^p(U^{an},\Omega_{U^{an},\log,0}^{N-p}))\otimes\mathbb C)^{\perp}\cap H_N^{\sing}(U^{an},\mathbb Q) \, \mbox{(D)}
\end{eqnarray*}
that is if for $\gamma\in H_N^{\sing}(U^{an},\mathbb Q)$ we have $ev(U)(\xi)(\gamma)=0$ for all $\xi\in H_{usu}^p(U^{an},\Omega_{U^{an},\log,0}^{N-p}\otimes\mathbb C)$
then $ev(U)(\omega)(\gamma)=0$ for all $\omega\in F^{p+1}H_{DR}(U)$, which is equivalent to 
\begin{equation*}
F^{p+1}H^N(U^{an},\mathbb Q)=OL^{N}_{U^{an},0}(H_{usu}^p(U^{an},\Omega_{U^{an},\log,0}^{N-p})).
\end{equation*}
Indeed if $\omega\neq 0\in F^{p+1}H^N(U^{an},\mathbb Q)\backslash OL^{N}_{U^{an},0}(H_{usu}^p(U^{an},\Omega_{U^{an},\log,0}^{N-p}))$, there would exist
$\gamma\in H_N^{\sing}(U^{an},\mathbb Q)$ such that 
\begin{equation*}
ev(U)(\omega)(\gamma)\neq 0  \; \mbox{ and } \; ev(U)(OL^{N}_{U^{an},0}(H_{usu}^p(U^{an},\Omega_{U^{an},\log,0}^{N-p})))(\gamma)=0, 
\end{equation*}
hence (D) does not hold. 
On the other hand, if there exist $\gamma\in H_N^{\sing}(U^{an},\mathbb Q)$ such that 
\begin{equation*}
ev(U)(OL^{N}_{U^{an},0}(H_{usu}^p(U^{an},\Omega_{U^{an},\log,0}^{N-p}))\otimes\mathbb C)(\gamma)=0 \; \mbox{ and } \;  ev(U)(F^{p+1}H_{DR}(U))(\gamma)\neq 0, 
\end{equation*}
considering a splitting 
\begin{equation*}
H_N^{\sing}(U^{an},\mathbb Q)=ev(U)(F^{p+1}H_{DR}(U))^{\perp}\cap H_N^{\sing}(U^{an},\mathbb Q)\oplus L,
\end{equation*}
there exists $\omega'\in F^{p+1}H^N(U^{an},\mathbb Q)$ such that $ev(U)(\omega')(\gamma)\neq 0$, giving 
\begin{equation*}
\omega'\neq 0\in F^{p+1}H^N(U^{an},\mathbb Q)\backslash OL^{N}_{U^{an},0}(H_{usu}^p(U^{an},\Omega_{U^{an},\log,0}^{N-p})).
\end{equation*}
Let $\mathbb P^{N,an}=\cup_{i=1}^nB_i$ an open cover, with $\phi_i:B_i\simeq D(0,1)^N$ open balls such that 
\begin{equation*}
\phi_i(X\cap B_i)=\left\{0\right\}\times D(0,1)^{N-1} \; \mbox{ and } \; \phi_i(z):=\phi_i(z_0,z_1,\ldots,z_N)=(z_0,\ldots,f(z),\ldots,z_N). 
\end{equation*}
Note that since $\mathbb P^{N,an}$ is compact, up to take finite subcover, we can assume that this cover is finite.
Denote, for $I\subset[1,\ldots,n]$, $l_I:B_I\hookrightarrow\mathbb P^N$ the open embeddings, with $B_I:=\cap_{i\in I}B_i$. 
Denote for $a,b\in\mathbb N$, $(a,b)$ the set of subsets of $b$ elements in a set of $a$ elements.
Then, using lemma \ref{OBac}, 
\begin{eqnarray*}
w=[B(w)]\in F^{p+1}H^N_{DR}(U), \\
B(w)=(w_I)=(\sum_{\nu\in(2p+1,p-1)}dh_{I,\nu}\wedge dz_{\nu_1}\wedge\cdots\wedge dz_{\nu_{p-1}})\wedge df/f
\in\Gamma(B_{\bullet},\Omega_{\mathbb P^N}(\log X)^{p+1,\partial=0})^{\partial_{\bullet}=0}, \\
h_{I,\nu}\in O(B_I), \; I\in(n,p),
\end{eqnarray*} 
which gives $w=(w_I)=\in\Gamma(B_{\bullet}\backslash X,\Omega_{U,\log,0}^{p+1})$ since on open balls all closed forms are exact and logarithmic (see remark \ref{ballclog}).
The evaluation map
\begin{eqnarray*}
ev(U):=ev(U)^{\bullet}_{\bullet}:\Gamma(B_{\bullet}\backslash X^{an},\Omega_{\mathbb P^{N,an}}^{\bullet})\to
C_{\diff}^{\bullet}(B_{\bullet}\backslash X^{an},\mathbb C),
\end{eqnarray*}
is a morphism of bi-complexes. On the other hand, we have the following canonical map
\begin{eqnarray*}
BC_{\partial}:H^p\Gamma(B_{\bullet}\backslash X^{an},C_{U^{an},\diff}^{N-p,\partial})
\xrightarrow{\Gamma(B_{\bullet}\backslash X^{an},\iota_{\partial})}
H^N\Gamma(B_{\bullet}\backslash X^{an},C_{U^{an},\diff}^{\bullet\geq N-p})
\xrightarrow{\Gamma(B_{\bullet}\backslash X^{an},\iota^{\geq N-p})} \\
H^N\Gamma(B_{\bullet}\backslash X^{an},C_{U^{an},\diff}^{\bullet})
\xrightarrow{(H^N(l_i^*=:C_{U^{an},\diff}^{\bullet}(l_i))_{1\leq i\leq n})^{-1}}
H^N_{\sing}(U^{an},\mathbb Q)=H^N\Gamma(U^{an},C_{U^{an},\diff}^{\bullet}),
\end{eqnarray*} 
where $H^N(l_i^*)$ is an isomorphism by the acyclicity of 
$C_{U^{an},\diff}^{j}\in\PSh(U^{an})$ for each $j\in\mathbb Z$, and 
\begin{equation*}
C_{U^{an},\diff}^{\bullet}\hookrightarrow C_{U^{an},\sing}^{\bullet}, \; \mbox{for} \; V\subset U^{an}, \; 
C^{\bullet}_{\diff}(V)\hookrightarrow C^{\bullet}_{\sing}(V) 
\end{equation*}
is the canonical quasi-isomorphism in $C(U^{an})$. We then have the following commutative diagram
\begin{equation*}
\xymatrix{
H^N\iota^{\geq p}H^N\Gamma(\mathbb P^N,\Omega^{\bullet}_{\mathbb P^N}\otimes_{O_{\mathbb P^N}}F^{\bullet-p}j_*O_U)
\ar[rrd]^{B_{\partial}}\ar[d]^{H^N((l_i^*=\Omega(l_i))_{1\leq i\leq n})} & \, & \\
H^N\iota^{\geq p}H^N\Gamma(B_{\bullet},(\Omega^{\bullet}_{\mathbb P^N}\otimes_{O_{\mathbb P^N}}F^{\bullet-p}j_*O_U)^{an})
\ar[rr]^{B_{\partial}}\ar[d]^{ev(U)} & \, & 
H^p\Gamma(B_{\bullet},\Omega^{N-p}_{\mathbb P^{N,an}}(\log X)^{\partial=0})\ar[d]^{ev(U)^{\bullet}_{\bullet}} \\
H^N_{\sing}(U^{an},\mathbb C)=H^N\Gamma(U^{an},C_{U^{an},\diff}^{\bullet}\otimes\mathbb C) & \, & 
H^p\Gamma(B_{\bullet}\backslash X^{an},C_{\diff}^{N-p,\partial=0}\otimes\mathbb C)
\ar[ll]^{BC_{\partial}\otimes\mathbb C}}
\end{equation*}
Since $ev(U)$ is injective by the strictness of the Hodge filtration and since $B_{\partial}$ is an isomorphism,
$ev(U)^{\bullet}_{\bullet}$ is injective and 
\begin{eqnarray*}
(BC_{\partial}\otimes\mathbb C)':=(BC_{\partial}\otimes\mathbb C)_{|ev(U)^{\bullet}_{\bullet}
(H^p\Gamma(B_{\bullet},\Omega^{N-p}_{\mathbb P^{N,an}}(\log X)^{\partial=0}))}: \\ 
ev(U)^{\bullet}_{\bullet}(H^p\Gamma(B_{\bullet},\Omega^{N-p}_{\mathbb P^{N,an}}(\log X)^{\partial=0}))
\to H^N_{\sing}(U^{an},\mathbb C)
\end{eqnarray*}
is injective. Recall that $H_N^{\sing}(U^{an},\mathbb Q)=H_NC^{\diff}_{\bullet}(B_{\bullet}\backslash X^{an},\mathbb Q)$ by the small chain theorem for topological spaces.
Then, assertion (D) become, using lemma \ref{OBac} : 
for each $\gamma=(\gamma_I)\in C^{diff}_{\bullet}(B_{\bullet}\backslash X,\mathbb Q)^{\partial+\partial_{\bullet}=0}$ such that $ev(U)(\xi)(\gamma)=0$ for all
\begin{equation*}
\xi=(\xi_I)\in\Gamma(B_{\bullet}\backslash X,\Omega^{p+1}_{\mathbb P^N,\log,0}\otimes\mathbb C)^{\partial_{\bullet}=0},
\end{equation*}
then $ev(U)(w)(\gamma)=0$ for all $w=(w_I)\in\Gamma(B_{\bullet},\Omega_{\mathbb P^N}^{p+1,\partial=0}(\log X))^{\partial_{\bullet}=0}$. 
But this follows from lemma \ref{Dlem}(ii). This proves the equality 
\begin{equation*}
F^{p+1}H^N(U^{an},\mathbb Q)=OL^{N}_{U^{an},0}(H_{usu}^p(U^{an},\Omega_{U^{an},\log,0}^{N-p})).
\end{equation*}

\end{proof}

\section{Hodge Conjecture for complex smooth projective hypersurfaces}

We now state and prove using the result of section 3 and section 4 the main result of this article :

\begin{thm}
Let $X=V(f)\subset\mathbb P^N_{\mathbb C}$ be a smooth projective hypersurface, with $N=2p+1$ odd.
Let $\lambda\in F^pH^{2p}(X^{an},\mathbb Q)$. Then there exists an algebraic cycle $Z\in\mathbb Z^p(X)$
such that $\lambda=[Z]$.
\end{thm}

\begin{proof}
Let $U:=\mathbb P^N_{\mathbb C}\backslash X$. 
Consider $Res_{X,\mathbb P^N}:M(X)\to M(U)[-1]$ the canonical map given in definition \ref{ResMot}. 
We have the commuative diagram of complex vector spaces : 
\begin{equation*}
\xymatrix{
H_{DR}^N(U)\ar[rrrr]^{H^N\Omega(Res_{X,\mathbb P^N})} & \, & \, & \, & H_{DR}^{N-1}(X)_v \\
H_{\sing}^N(U^{an},\mathbb C)
\ar[u]^{H^N\alpha(U^{an})=ev(U)^{-1}}\ar[rrrr]^{(H^N\Bti(Res_{X,\mathbb P^N}))\otimes_{\mathbb Q}\mathbb C} 
& \, & \, & \, & H_{\sing}^{N-1}(X^{an},\mathbb C)_v\ar[u]_{H^{N-1}\alpha(X^{an})=ev(X)^{-1}}}.
\end{equation*}
whose arrows are isomorphisms. Denote $w(\lambda):=H^{N-1}\alpha(X^{an})(\lambda)\in F^pH^{2p}_{DR}(X)$.
We have 
\begin{equation*}
\lambda=H^{N-1}\Bti(Res_{X,\mathbb P^N})(\alpha)+m[H], \; 
\mbox{with} \; m\in\mathbb Q, \; \alpha\in H_{\sing}^N(U^{an},\mathbb Q), 
\end{equation*}
where $H:=X\cap\Lambda\subset X$ is an hyperplane section. Then 
\begin{equation*}
w(\lambda)=H^{N-1}\Omega(Res_{X,\mathbb P^N})(w(\alpha))+m[H], \; \mbox{with} \;
w(\alpha):=H^N\alpha(U^{an})(\alpha)\in F^{p+1}H^N_{DR}(U).
\end{equation*}
Since $ev(U)(w(\alpha))=\alpha\in H_{\sing}^N(U^{an},2i\pi\mathbb Q)$, we have by proposition \ref{HUQ}, 
\begin{equation*}
w(\alpha)=H^NOL_{U^{an},0}(w(\alpha))\in H^NOL_{U^{an},0}(H_{usu}^p(U^{an},\Omega^{N-p}_{U^{an},\log,0}))
\end{equation*}
Hence, as in the proof of proposition \ref{cGAGAlog}(ii), we have by lemma \ref{a1trcan}(ii)', proposition \ref{UXusu} (i)' and (i)'', and lemma \ref{cGAGAloglem},
\begin{equation*}
w(\lambda)=H^{N-1}OL_{X^{an},0}(\Omega_{\log,0}^{N-p-1,an}(Res_{X,\mathbb P^N})(w(\alpha)))+m[H]\in 
OL_{X^{an},0}(H_{usu}^p(X^{an},\Omega^p_{X^{an},\log,0})). 
\end{equation*}
Hence by proposition \ref{cGAGAlog}(ii), $w(\lambda)=[Z]$ with $Z\in\mathcal Z^p(X)$. 
\end{proof}


LAGA UMR CNRS 7539 \\
Universit\'e Paris 13, Sorbonne Paris-Nord, 99 av Jean-Baptiste Clement, \\
93430 Villetaneuse, France, \\
bouali@math.univ-paris13.fr

\end{document}